\documentclass{amsart}
\usepackage{amssymb}
\usepackage{verbatim}
\usepackage{epic,eepic}

\newcommand{\Lap}{\Delta}
\newcommand\im{\operatorname{Im}}

\newcommand\NN{\mathbb{N}}
\newcommand\RR{\mathbb{R}}

\newcommand\Cx{\mathbb{C}}
\newcommand\sphere{\mathbb{S}}

\newcommand\dist{{\mathcal C}^{-\infty}}
\newcommand\CI{\mathcal{C}^\infty}
\newcommand\CIc{\mathcal{C}_c^\infty}
\newcommand\dCI{\dot{\mathcal{C}}^\infty}
\newcommand\Diff{\operatorname{Diff}}
\newcommand\pa{\partial}
\newcommand\supp{\operatorname{supp}}
\newcommand\ep{\epsilon}
\newcommand\bl{{\mathrm b}}

\newcommand\Vf{\mathcal{V}}
\newcommand\Vb{\mathcal{V}_\bl}

\newcommand\Diffb{\mathrm{Diff}_\bl}

\newcommand\Tb{{}^{\bl} T}

\newcommand\mf{\mathrm{mf}}
\newcommand\tf{\mathrm{tf}}
\newcommand\spf{\mathrm{spf}}

\newcommand{\E}{\mathcal{E}}

\newcommand\scl{{\mathrm{sc}}}
\newcommand\Tsc{{}^{\scl} T}
\newcommand\Vsc{\mathcal{V}_\scl}
\newcommand\Diffsc{\mathrm{Diff}_\scl}

\newcommand\Hb{H_\bl}

\newcommand\cL{\mathcal L}

\renewcommand{\Re}{\operatorname{Re}}

\newcommand{\abs}[1]{{\left\lvert{#1}\right\rvert}}
\newcommand{\norm}[1]{{\left\lVert{#1}\right\rVert}}

\newcommand{\ang}[1]{{\left\langle{#1}\right\rangle}}

\newcommand{\pA}{A'}
\newcommand{\ppA}{A''}

\newcommand{\tOmega}{\Omega}

\newcommand{\foobar}{s}
\newcommand{\decay}{\nu}
\newcommand{\dummyvar}{\mu}

\newcommand{\cutoff}{\vartheta}

\newtheorem{lemma}{Lemma}[section]
\newtheorem{prop}[lemma]{Proposition}
\newtheorem{thm}[lemma]{Theorem}
\newtheorem{cor}[lemma]{Corollary}

\newtheorem*{thm1prime}{Theorem 1.1'}
\newtheorem*{thm*}{Theorem}
\newtheorem*{prop*}{Proposition}
\newtheorem*{cor*}{Corollary}
\newtheorem*{conj*}{Conjecture}
\numberwithin{equation}{section}
\theoremstyle{remark}

\newtheorem*{rem*}{Remark}
\theoremstyle{definition}

\newtheorem*{Def*}{Definition}

\title{Morawetz estimates for the wave equation at low frequency}
\author[Andras Vasy]{Andr\'as Vasy}
\author{Jared Wunsch}
\address{Department of Mathematics, Stanford University}
\address{Department of Mathematics, Northwestern University}
\email{andras@math.stanford.edu}
\email{jwunsch@math.northwestern.edu}
\thanks{The authors gratefully
  acknowledge partial support from the NSF under grant numbers  
  DMS-0801226 (AV) and DMS-0700318, DMS-1001463 (JW). A.V. is also
grateful for support from a Chambers Fellowship at Stanford University.
The authors are grateful to Daniel Tataru, and to an anonymous referee, for
many helpful comments, and in particular for suggesting the refined
$\ell^\infty$--$\ell^1$ version of the estimates in Theorem~\ref{thm1}.}
\date{\today}

\begin{document}

\begin{abstract}
  We consider Morawetz estimates for weighted energy decay of
  solutions to the wave equation on scattering manifolds (i.e., those
  with large conic ends).  We show that a Morawetz estimate persists
  for solutions that are localized at low frequencies, independent of
  the geometry of the compact part of the manifold.  We further prove
  a new type of Morawetz estimate in this context, with both
  hypotheses and conclusion localized inside the forward light cone.
  This result allows us to gain a $1/2$ power of $t$ decay relative to
  what would be dictated by energy estimates, in a small part of
  spacetime.
\end{abstract}

\maketitle

\section{Introduction}

In this paper, we show that the celebrated Morawetz estimate
\cite{Morawetz}, expressing dispersion of solutions to the wave
equation, holds for low-frequency solutions on a wide class of
manifolds with asymptotically flat ends.  We also generalize the
estimate to a local-in-spacetime estimate inside the forward light
cone.

It is well known that the decay of energy of a solution to $\Box u =0$
at \emph{high frequencies} is closely tied to the geometry of geodesic
rays; in particular, the existence of trapped geodesics is an
obstruction to the uniform decay of local energy (see Ralston
\cite{Ralston}).  Many subsequent results have demonstrated that the
decay of high frequency components of the solution persists in a wide
variety of geometric settings in which there is no trapping of rays
(see e.g.\ \cite{M-R-S}, \cite{Vainberg}, \cite{Vodev}).

On the other hand, the \emph{low-frequency} behavior of the solutions
is both more robust and less studied.  Intuitively, long wavelengths
should be sensitive only to the crudest aspects of the geometric
setting, and in particular, trapping should not be an obstacle to
decay.
In this paper, we work in the setting of \emph{scattering
  manifolds,}
i.e., we assume that the manifold\footnote{We remark that in the body
  of the paper, $X$ will refer to the manifold with boundary given by
  \emph{compactifying} such a space.} $X$ has ends resembling the large ends of cones, with the
metric taking the form
\begin{equation}\label{scmetric1}
g= dr^2 + r^2 h(r^{-1}, d y)
\end{equation}
where the noncompact ends are diffeomorphic to $(1,\infty)_r \times
Y,$ with $Y$ a smooth manifold, and where $h$ is a family (in $r$) of
metrics on $Y.$\footnote{We may a priori assume that $h$ is in fact a
  smooth tensor including $dr$ components; that we may then change
  variables to remove these components is a result of Joshi-S\'a
  Barreto \cite{Joshi-SaBarreto}, Section 2.}  Our two main theorems
demonstrate the insensitivity of the low-energy behavior of the wave
equation to geometry on all but the largest scales.
We also allow long-range perturbations of the metrics \eqref{scmetric1}; see
Section~\ref{sec:conjugated} for the precise definitions.

The results in this paper are as follows.  To begin with, we state a
result holding for low frequency solutions in spatial dimension $n
\geq 4$ that generalizes the standard Morawetz estimate.  We fix a
cutoff function $\psi \in \CIc((0,\infty)).$ Let $\Psi_H$ denote the
frequency-localization operator $$\Psi_H = \psi(H^2(\Lap_g+V)).$$ We
assume that the function $r,$ appearing in the end structure of the
metric \eqref{scmetric1}, is extended to be a globally defined, smooth
function on $X,$ with $\min(r) \geq 2$ (so we may take $\log r$ with
impunity).  Let $\chi_0,$ $\chi_1$ be a partition of unity with
$\chi_1$ supported in $r >R \gg 0,$ hence in the ends in which the
product decomposition \eqref{scmetric1} applies.  Here and throughout
the paper, $\norm{\bullet}$ will refer to the norm in the space
$L^2(X, dg).$  When we employ spacetime norms in \S\ref{sec:local} we
will employ distinct notation $\norm{\bullet}_M$ where $M$ is (the
compactification of) the spacetime $\RR\times X.$
\begin{thm}\label{thm1}
Let $X$ be a scattering manifold of dimension $n \geq 4$ with a long-range metric,
as in \eqref{eq:metric-form}--\eqref{eq:g-1-form}, and let $V\geq 0$ be a
symbol of order $-2-\decay,$ $\decay>0$.  Let $u$ be a solution to the inhomogeneous wave
equation
$$
(\Box +V) u=f,\quad f\in L^1(\RR; L^2(X))
$$
on $\RR_t\times X,$ with initial data
$$
u|_{t=0}=u_0 \in H^1,\quad \pa_t u|_{t=0} =v_0\in L^2.
$$
Then for $H$ sufficiently large,
\begin{equation}\label{u-est-largedim}
\begin{aligned}
\int_0^\infty &\Big(\norm{r^{-3/2} \Psi_H u}^2 + \norm{\chi_1 (\log r)^{-1} r^{-1/2} \pa_r \Psi_H
  u}^2
\\ &\qquad\qquad+\norm{\chi_1r^{-3/2} \nabla_Y \Psi_H u}^2  + \norm{\chi_0
  \nabla_g \Psi_H u}^2\Big)\,dt\\
& \lesssim \E(0)+ \big( \int_0^\infty \norm{f} \, dt\big)^2,
\end{aligned}
\end{equation}
where $\E(0)$ is the initial energy
$$
\E(0) =  \norm{\Psi_H u_0}^2_{H^1}+\norm{\Psi_H v_0}^2_{L^2}.
$$
\end{thm}
Here $H^1$ and $L^2$ denote the appropriately defined spaces on $X$
(and $\norm{\bullet}$ denotes the $L^2(X)$ norm):
$L^2$ is the space of functions square integrable with respect to the
metric, and $H^1,$ which will henceforth be written $H^1_\scl$ (the
``scattering'' Sobolev space, associated to the geometry of large
conic ends) denotes
the space of functions satisfying
$$
\int \abs{u}^2 + \abs{\nabla_g u}_g^2 \, dg < \infty.
$$
The Laplacian $\Lap_g$ is the nonnegative Laplacian, and $\psi(H^2\Lap_g)$
is localizing at low energies.  The result thus shows that the
dispersive effects of the large scale scattering geometry always hold for
low-frequency solutions, regardless of trapping or other local
features.  We further remark that we could replace the $L^1(\RR;  L^2(X))$ norm on $f$ on the RHS of
the estimate by
\begin{equation}\label{eq:log-forcing}
\int_0^\infty \norm{(r^2+t^2)^{1/4}
 \log (r^2+t^2) f}^2 \, dt,
\end{equation}
if desired, to obtain a weighted $L^2$ spacetime norm instead.

Following the suggestion of an anonymous referee, we are in fact able to
prove a refined version of this result, with logarithmic losses replaced by
estimates in $\ell^1$ and $\ell^\infty$ norms on energy in dyadic spatial
shells.  Let $\Upsilon_k$ denote a spatial decomposition in radial
dyadic
shells (to be discussed further in \S\ref{sec:Morawetz}, see
\eqref{eq:Ups-def}), so $r\sim 2^k$ on $\Upsilon_k$.
\begin{thm1prime}\label{thm1prime}
Let $X$ be a scattering manifold of dimension $n \geq 4$ with a long-range metric,
as in \eqref{eq:metric-form}--\eqref{eq:g-1-form}, and let $V\geq 0$ be a
symbol of order $-2-\decay,$ $\decay>0$.  Let $u$ be a solution to the inhomogeneous wave
equation
$$
(\Box +V) u=f,\quad f\in L^1(\RR; L^2(X)) \cap \ell^1(\NN_k; L^2(\RR
\times \Upsilon_k))
$$
on $\RR_t\times X,$ with initial data
$$
u|_{t=0}=u_0 \in H^1,\quad \pa_t u|_{t=0} =v_0\in L^2.
$$
Then for $H$ sufficiently large,
\begin{equation}\label{u-est-largedim-prime}
\begin{aligned}
\int_0^\infty &\Big(\norm{r^{-3/2} \Psi_H u}^2 +\norm{\chi_1r^{-3/2}
  \nabla_Y \Psi_H u}^2
  + \norm{\chi_0
  \nabla_g \Psi_H u}^2\Big)\,dt\\
&\qquad\qquad+\norm{\chi_1 r^{-1/2}\pa_r
 \Psi_H u }_{\ell^\infty(\NN_k; L^2([0,\infty))\times \Upsilon_k))}^2 \\
& \lesssim \E(0)+ \norm{f}_{L^1([0,\infty); L^2(X))}^2 + \norm{r^{1/2}f}_{\ell^1(\NN_k; L^2([0,\infty)
\times \Upsilon_k))}^2
\end{aligned}
\end{equation}
where $\E(0)$ is the initial energy
$$
\E(0) =  \norm{\Psi_H u_0}^2_{H^1}+\norm{\Psi_H v_0}^2_{L^2}.
$$
\end{thm1prime}

Here we could drop $\int_0^\infty\norm{\chi_1 (\log r)^{-1} r^{-1/2} \pa_r \Psi_H
 u}^2\,dt$ from the left hand side of \eqref{u-est-largedim-prime} as
compared to
\eqref{u-est-largedim}
without any loss, since it
can be estimated by
\begin{equation*}\begin{aligned}
\int_0^\infty&\norm{\chi_1 (\log r)^{-1} r^{-1/2} \pa_r \Psi_H
 u}^2\,dt=\sum_k\int_0^\infty\int_{\Upsilon_k}\chi_1^2 (\log r)^{-2} r^{-1} |\pa_r \Psi_H
 u|^2\,dg\,dt\\
&\lesssim  \Big(\sum_k k^{-2}\Big)
 \norm{\chi_1 r^{-1/2}\pa_r
 \Psi_H u }_{\ell^\infty(\NN; L^2([0,\infty))\times \Upsilon_k))}^2,
\end{aligned}\end{equation*}
so the quantity we estimate on the left hand side in
Theorem~1.1' indeed
stronger than that in Theorem~\ref{thm1}. A similar argument shows
that \eqref{eq:log-forcing} also dominates the third term on the
right hand side of \eqref{u-est-largedim-prime}.

Our second main result is a version of the Morawetz estimate that is
localized in spacetime, with both hypotheses and conclusion localized
in a sub-cone of the forward light cone.  It holds in all dimensions
$n\geq 3.$

For $\delta>0$ let \begin{equation}\label{omegadef}\tOmega=\tOmega_{\delta} =\{t>1/\delta,\
r/t<\delta\}\subset \RR \times X.\end{equation}
This is the (asymptotic) cone in which we will localize.

In the statement of the following theorem, $\nabla$ denotes the
space-time gradient.
\begin{thm}\label{thm2}
Let $X$ be a scattering manifold of dimension $n \geq 3$ with
a long-range metric as in \eqref{eq:metric-form}--\eqref{eq:g-1-form}, $V\geq 0$ a
symbol of order $-2-\decay,$ $\decay>0$.  Let
$$
(\Box +V) u=f.
$$
Suppose that $\delta<1$ and
$$
\nabla \Psi_H u\in t^\kappa L^2(\tOmega),\ f\in r^{-1/2}t^{\kappa-1/2} L^2(\tOmega),
\ \kappa\in\RR.
$$
Then for $H$ sufficiently large,
$$
\nabla \Psi_H u\in t^\kappa(r/t)^\sigma L^2(\tOmega),
\ \Psi_H u\in t^\kappa r (r/t)^\sigma L^2(\tOmega)
$$
for $0<\sigma<1/2$.

In particular, if $K$ is compact in $X^\circ$, then on $\RR\times K$,
$\Psi_H u,\nabla \Psi_H u$ are in $t^{\kappa-1/2+\ep}L^2(\RR\times X)$ for every $\ep>0$ (for $H$
sufficiently large).
\end{thm}

This theorem gives almost half an order of decay as compared to the a
priori assumption, but only in a small part of space-time. Indeed, for
any $\delta_0>0$, the conclusion in $r/t>\delta_0$ is the same as
hypothesis, so in particular $\delta$ in the statement of the theorem
can be taken arbitrarily small without losing the strength of the
conclusion. Thus, Theorem~\ref{thm2} really amounts to a regularity
statement; it is a close analogue of the statement for hyperbolic PDE
that over compact sets, one has extra regularity over a priori
space-time Sobolev regularity (square integrability can be replaced by
continuity along the flow). Indeed, an analogous statement is that
locally $L^2$ solutions of the wave equation in space time, near
$t=0$, are in $t^\sigma L^2$ locally for $\sigma\in(0,1/2)$; this
follows from continuity in time with values in $L^2$ and the fact that
$t^{-\sigma}$ is in $L^2(\RR)$ locally.

One virtue of this version of
the Morawetz estimate, and even more of the approach to obtaining this estimate,
is that they hold the promise of applying quite
broadly, in non-product spacetimes. Even in the present inhomogeneous form they
can be used to study the wave equation for
metrics which arise by perturbing the Minkowski metric
near infinity in the following way: pick $p_j$, $j=1,\ldots,N$, on the sphere at
infinity, $\sphere^n=\pa\overline{\RR^{n+1}_{t,z}}$,
in the radial compactification of $\RR^{n+1}=\RR_t\times\RR^n_z$
which lie in the interior of
the forward light cone. Blow up the $p_j$; if one of the $p_j$ is the ``north pole,''
corresponding to $z=0$, $t=+\infty$, a neighborhood of the front face can be
identified with a neighborhood of ``temporal infinity'' in the product asymptotically
Euclidean space-time (see Section~\ref{sec:local}); the other $p_j$ can be transformed
to this via isometries of Minkowski space-time. Thus, if the Minkowski metric
is perturbed in a way corresponding to the local product structure these transformations
induce, modulo decaying terms in time, the local Morawetz estimate is applicable
by treating the decaying terms as inhomogeneities.
In future work, we
plan to address such non-product geometries more systematically.

By the standard energy estimate, if $(\Box+V) u=0$ and $u$ is in the
energy space, then for $\phi\in\CI_c(\RR)$,
$$
\phi(r/t)\nabla u\in t^{1/2+\delta}L^2(\RR\times X),\ \delta>0.
$$
This observation immediately yields
\begin{cor}\label{cor1}
  Let $X,$ $g,$ and $V$ be as in Theorem~\ref{thm2}.  Suppose that $\Box
  u=0,$ with initial conditions
$$
u|_{t=0}=u_0 \in H^1,\quad \pa_t u|_{t=0} =v_0\in L^2.
$$
There exists $\delta>0$ small such that
\begin{equation}
\iint_{\tOmega_\delta} \abs{t^{\sigma-1/2-\delta} r^{-1-\sigma}
  \Psi_H u}^2+\abs{t^{\sigma-1/2-\delta} r^{-\sigma} \nabla
 \Psi_H u}^2
 \, dt \, dg \lesssim \E(0)
\end{equation}
whenever $\sigma \in (0,1/2),$ and $H \gg 0$ is sufficiently large.
\end{cor}
The result is of special interest in the case $n=3,$ where
Theorem~\ref{thm1} does not apply, but where Corollary~\ref{cor1} yields
a spatially localized estimate.  Let $K \subset X^\circ$ be compact.
Then Corollary~\ref{cor1} specializes to show:
\begin{cor}
For all $\ep>0,$
$$
\int_0^\infty \norm{t^{-\ep}\Psi_H
  u}^2+\norm{t^{-\ep}\nabla\Psi_H u}^2 \, dt\lesssim \E(0).
$$
for $H$ sufficiently large.
\end{cor}
Note, by contrast, that conservation of energy would give the same
estimate with $t^{-\ep}$ replaced by $t^{-1/2-\ep},$ so this estimate
represents an improvement of $t^{-1/2}$ in decay relative to energy
estimates; by contrast, it loses $t^{-\ep}$ relative to
Theorem~\ref{thm2} (which of course only holds for $n\geq 4$).

We emphasize that Theorem~\ref{thm2} and its corollaries can
be proved without the
need for localization at low frequencies so long as we can find a
Morawetz commutant that works \emph{globally}, and in particular this
is the case on small perturbations of Euclidean space.

The proof of Theorem~\ref{thm1} is, as with the usual Morawetz inequality, a ``multiplier''
argument based on the first order differential operator
$$
A_0 = (1/2)(\pa_r - \pa_r^*).
$$
The subtlety is that the crucial commutator term
$$
[\Lap,A_0]
$$
is positive for $r \gg 0,$ in the ends of the manifold, but certainly
not in the interior where there is indeed no natural definition of the
radial function $r$ or of $\pa_r.$ We must thus introduce cutoffs,
which in turn introduce terms in the commutator that have to be
treated as errors.  At high frequency, these errors are disastrous,
but at low frequency, we may employ a Poincar\'e/Hardy inequality to
control them.  Further technical subtleties arise in estimating
remainder terms from the commutator, applied to $\Psi_H u.$ These
estimates are among the principal technical innovations here---see
Proposition~\ref{prop:conjugated-cutoff} below.

The proof of Theorem~\ref{thm2} uses a slightly different commutator
argument (formally a ``commutator'' rather than a ``multiplier'' in
the usual parlance) that occurs in space-time.  The localizer in the
cone $\tOmega$ multiplies a Morawetz-like commutant with a positive
weight in $r$ but decay in $t;$ derivatives of the cutoff are
controlled by the a priori decay assumptions, namely
$\nabla \Psi_H u|_{\tOmega} \in t^\kappa L^2(\RR\times X)$, which in the corollaries
are implied by energy estimates.

Related results have recently been pursued by a number of authors in the
setting of asymptotically Euclidean manifolds.  Bony-H\"afner
\cite{Bony-Haefner1} explored Mourre estimates on asymptotically Euclidean
spaces at low energy, while both Bouclet \cite{Bouclet} and Bony-H\"afner
\cite{Bony-Haefner3} have recently obtain iterated resolvent estimates in
the low frequency regime sufficient to prove strong weighted decay
estimates for asymptotically Euclidean metrics.  These results yield
stronger decay than what is obtained here, but in a narrower class of
geometries and at the cost of a more intricate argument which involves
first proving resolvent estimates; one of the virtues of the techniques
used here is the direct use of the hyperbolic equation, apart from the
estimates necessary to localize at low energy.  More in the spirit of our
approach, Metcalfe-Tataru have proved an analogous theorem to our
Theorem~\ref{thm1} for small perturbations of Euclidean space, and Tataru
\cite{Tataru} has also proved strong energy decay estimates in a class of
$3+1$-dimensional Lorentz metrics that includes the Schwarzschild and Kerr
spacetimes; see also the more recent \cite{MTT}.  (There has been
considerable work on these specific Lorentzian examples motivated by
problems in general relativity; we will not review that literature here.)
The only prior results on low energy estimates in the setting of general
scattering manifolds, in addition to the authors'
\cite{Vasy-Wunsch:Positive}, is the recent work of
Guillarmou-Hassell-Sikora \cite{G-H-S}, which employs a sophisticated
low-energy parametrix construction to obtain strong decay estimates.  Due
to the nature of the construction, the requirements on the metric in
\cite{G-H-S} are more stringent than the long-range assumptions used here.

Many of the low-frequency estimates used in this paper were first
developed by the authors in \cite{Vasy-Wunsch:Positive} following the
results of Bony-H\"afner \cite{Bony-Haefner1} in the asymptotically
Euclidean setting, using rather different methods.

We have discussed both low- and high-frequency estimates above; for
completeness, we remark that {\em intermediate frequencies} (i.e.\
frequencies in any compact subset of $(0,\infty)$) are always
well-behaved in that the solution of the wave equation localized to
these frequencies decays rapidly inside the forward light cone, in
$r/t<c<1$. This can be seen easily both by commutator methods
(especially tools analogous to unreduced 2-body type problems, i.e.\
Mourre estimates as in the work of Mourre \cite{Mourre:Operateurs},
Perry, Sigal and Simon \cite{Perry-Sigal-Simon:Spectral}), Froese and
Herbst \cite{FroMourre}, or directly using the resolvent of the
Laplacian at the continuous spectrum, as described by Hassell and the
first author in \cite{Hassell-Vasy:Spectral, Hassell-Vasy:Resolvent}
and the stationary phase lemma.

\section{Conjugated spectral cutoffs}\label{sec:conjugated}
\subsection{Notation and setting}
Before stating our results on conjugated spectral cutoffs, we very
briefly recall the basic definitions of the b- and scattering
structures on a compact $n$-dimensional manifolds with boundary,
denoted $X$; we refer to \cite{RBMSpec} for more detail.  A boundary
defining function $x$ on $X$ is a non-negative $\CI$ function on $X$
whose zero set is exactly $\pa X$, and whose differential does not
vanish there; a useful example to keep in mind is $x=r^{-1}$ (modified
to be smooth at the origin) in the compactification of an
asymptotically Euclidean space. We recall that $\dCI(X)$, which may
also be called the set of Schwartz functions, is the subset of
$\CI(X)$ consisting of functions vanishing at the boundary with all
derivatives, the dual of $\dCI(X)$ is tempered distributional
densities $\dist(X;\Omega X)$; tempered distributions $\dist(X)$ are
elements of the dual of Schwartz densities, $\dCI(X;\Omega X)$.

Let $\Vf(X)$ be the Lie algebra of all $\CI$ vector fields on $X$;
thus $\Vf(X)$ is the set of all $\CI$ sections of $TX$. In local
coordinates $(x,y_1,\ldots,y_{n-1})$,
$$
\pa_x,\pa_{y_1},\ldots,\pa_{y_{n-1}}
$$
form a local basis for $\Vf(X)$, i.e.\ restrictions of elements
of $\Vf(X)$ to the coordinate chart can be expressed uniquely as
a linear combination of these vector fields with $\CI$ coefficients.
We next define $\Vb(X)$ to be the Lie algebra of $\CI$ vector fields tangent
to $\pa X$; in local coordinates 
$$
x\pa_x,\pa_{y_1},\ldots,\pa_{y_{n-1}}
$$
form a local basis in the same sense. Thus, $\Vb(X)$ is the set
of all $\CI$ sections of a bundle, called the b-tangent bundle of
$X$, denoted $\Tb X$. Finally, $\Vsc(X)=x\Vb(X)$ is the Lie algebra
of scattering vector fields;
$$
x^2\pa_x,x\pa_{y_1},\ldots,x\pa_{y_{n-1}}
$$
form a local basis now. Again, $\Vsc(X)$ is the set of 
all $\CI$ sections of a bundle, called the scattering tangent bundle of
$X$, denoted $\Tsc X$. We write $\Diffb(X)$, resp.\ $\Diffsc(X)$, for the
algebra of differential operators generated by $\Vb(X)$, resp.\ $\Vsc(X)$,
over $\CI(X)$; these are thus finite sums of finite products of vector
fields in $\Vb(X)$, resp.\ $\Vsc(X)$, with $\CI(X)$ coefficients. The dual
bundles of $\Tb X$, resp.\ $\Tsc X$ are $\Tb^*X$ and $\Tsc^*X$, and are called
the b- and the scattering cotangent bundles, respectively; locally they are
spanned by
$$
\frac{dx}{x},\ d_{y_1},\ldots, dy_{n-1},\ \text{resp.}
\ \frac{dx}{x^2},\ \frac{d_{y_1}}{x},\ldots, \frac{dy_{n-1}}{x}.
$$

We let $S^k(X),$ the space of symbols of order $k$, consist of
functions $f$ such that
$$
x^k Lf\in L^\infty(X)\ \text{for all}\ L\in\Diffb(X).
$$
We note, in particular, that
$$
x^\rho\CI(X)\subset S^{-\rho}(X)
$$
since $\Diffb(X)\subset\Diff(X)$.
As $\Diffb(X)$ (a priori acting, say, on tempered distributions)
preserves $S^k(X)$, and one can extend $\Diffb(X)$ and $\Diffsc(X)$
by ``generalizing the coefficients:''
$$
S^k\Diffb^m(X)=\{\sum_j a_j Q_j:\ a_j\in S^k(X),\ Q_j\in\Diffb^m(X)\},
$$
with the sum being locally finite, and defining $S^k\Diffsc^m(X)$
similarly. In particular,
$$
x^k\Diffb^m(X)\subset S^{-k}\Diffb^m(X),
\ x^k\Diffsc^m(X)\subset S^{-k}\Diffsc^m(X).
$$
Then $Q\in S^{k}\Diffsc^m(X)$, $Q'\in S^{k'}\Diffsc^{m'}(X)$ gives
$QQ'\in S^{k+k'}\Diffsc^{m+m'}(X)$, and the analogous statement
for $S^k\Diffb^m(X)$ also holds.

In this paper, we are concerned with \emph{scattering metrics,} that
is to say, smooth metrics on $X^\circ$ that, near the boundary, take
the form
\begin{equation}\label{eq:metric-form}
g= \frac{dx^2}{x^4} + \frac{h}{x^2}+g_1
\end{equation}
with $h=h(x,y,dy)$ a family of metrics on $\pa X,$ and
\begin{equation}\label{eq:g-1-form}
g_1\in S^{-\decay}(X;\Tsc^*X\otimes\Tsc^*X)
\end{equation}
for some $\decay>0$ is symmetric. Thus,
$g$ is a positive definite inner product on the fibers of $\Tsc X$; long-range
metrics like this were considered in \cite{Vasy-Wunsch:Positive}.
We refer to Section~2 and the beginning of Section~3
of \cite{Vasy-Wunsch:Positive} for more details. Note that
substituting $r=x^{-1}$ gives form \ref{scmetric1} described above when $g_1=0$.
We let $L^2(X)$ denote the space of functions square integrable with
respect to the metric density; locally near a point in the boundary,
this is equivalent to using the density
$$
x^{-(n+1)} \,dx\, dy_1 \dots dy_{n-1}.
$$
Note that the vector fields in $\Vsc(X)$ are precisely those with
bounded length with respect to $g.$ Correspondingly, a typical example
of an element of $\Diffsc^2(X)$ is the Laplace-Beltrami operator,
$\Delta_g=d^* d$.  We also permit short range
potential perturbations of $\Delta_g$, namely
\begin{equation}\label{eq:V-form}
V\in S^{-2-\decay}(X),\ \decay>0,\ V\geq 0;
\end{equation}
We assume that $$V\geq 0,$$
though the arguments also work for potentials with small negative
parts; see \cite[Footnote~5]{Vasy-Wunsch:Positive}.

\subsection{Estimates}
In this section we employ some of the results of \cite{Vasy-Wunsch:Positive}
to obtain the following.  Let $\psi\in\CI_c(\RR)$ and
$$
\Psi_H=\psi(H^2(\Lap_g+V)),\qquad \Delta=\Delta_g.
$$
\begin{prop}\label{prop:conjugated-cutoff}
Suppose $n\geq 3$, $g$ and $V$ as in \eqref{eq:metric-form}, \eqref{eq:g-1-form},
\eqref{eq:V-form}.
For $0\leq s$, $0\leq \rho$, $s+\rho<\min(2,n/2)$,
$L\in \Diffsc^1(X)$,
$$
Lx^{s+\rho}\Psi_H x^{-s}\leq CH^{-\rho},\ H>1.
$$
\end{prop}
While the spectral cutoff is in the ``(anti-)semi-classical'' operator
$H^2(\Lap_g+V)$ with large parameter, we emphasize that the operator $L$
in the above proposition is simply a differential operator without
parameter.  As first-order scattering differential operators are
spanned over $\CI(X)$ by $x^2\pa_x,$ $\pa_y,$ and the constant
function, to prove the proposition, it suffices to check for these
particular values of $L.$  Equivalently, we can simply use the
vector-valued $L= \nabla_g,$ the gradient with respect to
the scattering metric, as well as $L=1.$

Note that for $\rho=0$ and $s=0$
Proposition~\ref{prop:conjugated-cutoff} follows automatically from
the functional calculus and elliptic regularity; for $\rho=0$ and
$s=1$ this is proved in \cite{Vasy-Wunsch:Positive}, hence it follows
by interpolation for $\rho=0$, $0\leq s\leq 1$. Thus, if $\rho=0$, we
need to deal with $1<s<2$ if $n\geq 4$, and $1<s<3/2$ if $n=3$.

Now, with $\tilde\psi\in\CI_c(\Cx)$ an almost analytic extension of $\psi$, and
$$
R(z)=(H^2(\Delta_g+V)-z)^{-1}
$$
we have
\begin{gather}\label{eq:func-calc}
Lx^{s+\rho}\Psi_Hx^{-s}=\frac{1}{2\pi}\int\overline{\pa}\tilde\psi(z)
Lx^{s+\rho} R(z)x^{-s} \,dz\,d\bar z.
\end{gather}
Thus, we only need to obtain uniform bounds on $Lx^{s+\rho} R(z)x^{-s}$ to prove
the proposition.

We begin with a preliminary lemma, which is an analogue of
\cite[Proposition~4.3]{Vasy-Wunsch:Positive}, where it was proved in
the range $0\leq s<1/2$ for $n\geq 3$.

\begin{lemma}\label{lemma:weight-up-to-1}
For $0\leq s<\min(1,(n-2)/2)$,
$$
\|x^s\nabla_g u\|\leq C\|(\Delta_g+V)u\|^{(1+s)/2}\|u\|^{(1-s)/2}
$$
\end{lemma}

\begin{proof}
By \cite[Proposition~4.3]{Vasy-Wunsch:Positive}, we may assume $n\geq 4$; then
the range in the lemma is $0\leq s<1$.
By \cite[Equation~(A.5)]{Vasy-Wunsch:Positive}, we have for $0\leq s<(n-2)/2$,
\begin{equation}\label{eq:convexity-est}
\|x^s\nabla_g u\|\leq C\|(\Delta_g+V)u\|^{1/2}\|x^{2s}u\|^{1/2}.
\end{equation} 
Moreover, by \cite[Corollary~3.5]{Vasy-Wunsch:Positive}, we have for
$0<s<(n-2)/2$, $0\leq \theta\leq 1$,
\begin{equation}\label{eq:gen-Poincare}
\|x^{s+\theta} u\|\leq C\|x^s\nabla u\|^\theta\|x^s u\|^{1-\theta}.
\end{equation}
Applying this with $0<\theta=s<1$, we obtain that
$$
\|x^{2s} u\|\leq C\|x^s\nabla_g u\|^s\|x^s u\|^{1-s}.
$$
Substituting into \eqref{eq:convexity-est} yields
\begin{equation}\label{eq:weighted-5}
\|x^s\nabla_g u\|^{1-s/2}\leq C\|(\Delta_g+V)u\|^{1/2}\|x^su\|^{(1-s)/2}.
\end{equation} 
Now using \eqref{eq:gen-Poincare} with $0$ in place of $s$ and $s$ in place of
$\theta$ yields
\begin{equation*}
\|x^{s} u\|\leq C\|\nabla_g u\|^s\|u\|^{1-s}\leq C\|(\Delta_g+V)u\|^{s/2}\|u\|^{1-s/2},
\end{equation*}
with the second inequality arising from
$$
\|\nabla_g u\|^2\leq \|(\Delta_g+V)u\|\,\|u\|,
$$
which is \eqref{eq:convexity-est} with $s=0.$
Substitution into \eqref{eq:weighted-5} yields
\begin{equation}\label{eq:weighted-8}
\|x^s\nabla_g u\|^{1-s/2}\leq C\|(\Delta_g+V)u\|^{1/2+s(1-s)/4}\|u\|^{(1-s)(1-s/2)/2}.
\end{equation}
Since $1/2+s(1-s)/4=(1-s/2)(1+s)/2$, raising both sides to the
power $(1-s/2)^{-1}$ yields
$$
\|x^s\nabla_g u\|\leq C\|(\Delta_g+V)u\|^{(1+s)/2}\|u\|^{(1-s)/2}.
$$
This finishes the proof.
\end{proof}

Now, letting $u=R(z)f$, and writing
$$
(\Delta_g+V)u=H^{-2}(H^2(\Delta_g+V)-z)u+H^{-2}zu,
$$
we deduce from Lemma~\ref{lemma:weight-up-to-1} that for
$0\leq s<\min(2,(n-2)/2)$,
$$
\|x^s\nabla_g R(z)f\|\leq C\|H^{-2}(f+zR(z)f)\|^{(1+s)/2}\|R(z)f\|^{(1-s)/2}.
$$
Since $\|R(z)f\|\leq |\im z|^{-1}\|f\|$, this
gives
\begin{gather}\label{eq:res-est-up-to-1}
\|x^s\nabla_g R(z)f\|\leq CH^{-(1+s)}(1+|z|/|\im z|)^{(1+s)/2}|\im z|^{-(1-s)/2}\|f\|\\
\leq CH^{-(1+s)}|z|^{(1+s)/2}|\im z|^{-1}\|f\|.
\end{gather}
Using the Hardy/Poincar\'e inequality \cite[Proposition 3.4]{Vasy-Wunsch:Positive}, we deduce
the following extension of \cite[Proposition~4.5]{Vasy-Wunsch:Positive} (where
it was proved for $0\leq s<1/2$, i.e.\ the result below is only an improvement
if $n\geq 4$):

\begin{cor}\label{cor:ext-res-bounds}
Suppose $0\leq s<\min(1,(n-2)/2)$, $L\in S^{-1-s}\Diffb^1(X)$. Then
$$
\|LR(z)f\|\leq CH^{-(1+s)}|z|^{(1+s)/2}|\im z|^{-1}\|f\|.
$$
\end{cor}

We can extend the range of $s$ when $L$ is zero'th order by interpolating
the case $s=0$ above, namely
$$
\|xR(z)f\|\leq CH^{-1}|z|^{1/2}|\im z|^{-1}\|f\|,
$$
with the estimate $\|R(z)f\|\leq |\im z|^{-1}\|f\|$ to obtain the following result
(in which the range $1\leq s<\min(2,n/2)$ is a special case of
Corollary~\ref{cor:ext-res-bounds}, and it is the range $0\leq s\leq
1$ that follows from the interpolation).

\begin{cor}\label{cor:ext-res-bounds-0th}
Suppose $0\leq s<\min(2,n/2)$. Then
$$
\|x^sR(z)f\|\leq CH^{-s}|z|^{s/2}|\im z|^{-1}\|f\|.
$$
\end{cor}

As we are interested in the functional
calculus for compactly supported functions,
we suppress the large $z$ behavior, and work with $z$ in a compact set.
As a consequence of the preceding two corollaries we conclude:

\begin{cor}\label{cor:ext-res-bounds-sc}
For $0\leq s<\min(2,n/2)$, $L\in S^{-s}\Diffsc^1(X)$, $z$ in a compact set,
\begin{equation}\label{eq:ext-res-bounds-sc}
\|LR(z)f\|\leq CH^{-s}|\im z|^{-1}\|f\|,
\end{equation}
uniformly in $z$ and $H\geq 1$.
\end{cor}

\begin{proof}
If $1\leq s$, then $L\in S^{-s}\Diffsc^1(X)\subset S^{-s}\Diffb^1(X)$ and
Corollary~\ref{cor:ext-res-bounds} proves the result.

On the other hand, if $0\leq s<1$, one can write $L=aV+b$, $a,b\in S^{-s}(X)$,
$V\in\Vsc(X)=x\Vb(X)$, so $aV\in S^{-1-s}\Diffb^1(X)$, so
Corollary~\ref{cor:ext-res-bounds} gives
$$
\|aVR(z)f\|\leq CH^{-1-s}|\im z|^{-1}\|f\|.
$$
Furthermore, Corollary~\ref{cor:ext-res-bounds-0th} yields
$$
\|bR(z)f\|\leq CH^{-s}|\im z|^{-1}\|f\|.
$$
Combining these two estimates
proves the corollary.
\end{proof}

Now we consider conjugates of $R(z)$. Below sometimes
$n=3$ and $n\geq 4$ are treated
separately because of the stronger limitations $n=3$ imposes in the preceding
results: $\min(2,n/2)=3/2$, rather than $2$, in that case.
First we prove a result for conjugation by small powers of $x$ (namely, $x^{s}$,
$0\leq s \leq 1$); then in Proposition~\ref{prop:res-est-up-to-2} we
increase the range. Roughly speaking, the argument of
Proposition~\ref{prop:res-est-up-to-2} reduces the weight by $1$ (and could
be done inductively if Corollaries~\ref{cor:ext-res-bounds}
and \ref{cor:ext-res-bounds-sc} had been proved
for an extended range of weights); Lemma~\ref{lemma:res-small-conj} deals
with the fractional part of the weight (and thus would be the base case in the
inductive argument).

\begin{lemma}\label{lemma:res-small-conj}
Suppose $n\geq 3$. Suppose  $0\leq\foobar\leq 1$,
$0\leq \rho$, $\foobar+\rho<\min(2,n/2)$, and if $n=3$ then in addition $\rho<1$.
For $L\in \Diffsc^1(X)$,
and $z$ in a compact set,
\begin{equation}\label{eq:res-est-8}
\|Lx^{\rho+\foobar} R(z)x^{-\foobar}\|\leq CH^{-\rho}|\im z|^{-2}
\end{equation}
uniformly as $H\to +\infty$.
\end{lemma}

\begin{proof}
The case $\foobar=0$ follows from Corollary~\ref{cor:ext-res-bounds-sc}.
Note that the conditions imply $0\leq \rho<2$ if $n\geq 4$.

If $n=3$, assume that $1/2<\foobar\leq 1$; if $n\geq 4$ no additional assumption
is made.
Then
\begin{gather*}
L x^{\foobar+\rho} R(z) x^{-\foobar}=Lx^{\rho} R(z)+Lx^{\foobar+\rho}[R(z),x^{-\foobar}]\\
=Lx^{\rho} R(z)-Lx^{\foobar+\rho}R(z)[H^2(\Delta+V),x^{-\foobar}]R(z).
\end{gather*}
For $L\in\Diffsc^1(X)$, by Corollary~\ref{cor:ext-res-bounds-sc}, as $0\leq \rho<\min(2,n/2)$,
\begin{gather*}
\|Lx^\rho R(z)\|\leq CH^{-\rho}|\im z|^{-1}
\end{gather*}
As for the second term
on the right hand side, we note that
$[\Delta+V,x^{-\foobar}]\in x^{2-\foobar}\Diffb^1(X).$  Consequently, we may employ
Corollary~\ref{cor:ext-res-bounds}, using $0<\foobar\leq 1$ (so $1\leq 2-\foobar<2$)
if $n\geq 4$, resp.\ $1/2<\foobar\leq 1$ (so $1\leq 2-\foobar<3/2$) if
$n=3$.  This yields
$$
\|[\Delta+V,x^{-\foobar}]R(z)\|\leq CH^{-(2-\foobar)}|\im z|^{-1}.
$$
On the other hand, as $0\leq \foobar+\rho<\min(2,n/2)$,
by Corollary~\ref{cor:ext-res-bounds-sc},
$$
\|Lx^{\foobar+\rho}R(z)\|\leq CH^{-(\foobar+\rho)}|\im z|^{-1},
$$
so
$$
\|Lx^{\foobar+\rho}R(z)[H^2(\Delta+V),x^{-\foobar}]R(z)\|\leq CH^{-\rho}|\im z|^{-2}.
$$
Combining these results proves \eqref{eq:res-est-8}
if $0<\foobar\leq 1$ and $n\geq 4$, and also if $1/2<\foobar\leq 1$ and $n=3$,
and completes
the proof in these cases.

It remains to deal with $n=3$, and $0<\foobar\leq 1/2$.
This follows by interpolation between $\foobar=0$ and
$\foobar>1/2$, noting that given $0\leq \rho<1$ we can take
$\foobar'\in(1/2,1)$ such that $\rho+\foobar'<3/2$; we then interpolate between $0$ and
this value $\foobar'$. This last step is the reason for the restriction $\rho<1$; if $\rho\geq 1$,
the desired $\foobar'$ does not exist. This completes the proof of the Lemma.
\end{proof}

We now extend the range of allowable exponents $s$ by an inductive argument:

\begin{prop}\label{prop:res-est-up-to-2}
Suppose $n\geq 3$.
Suppose $\rho\geq 0$, $L \in\Diffsc^1(X)$. For $0\leq s\leq s+\rho<\min(2,n/2)$,
\begin{equation}\label{eq:res-est-up-to-3-halves}
\|Lx^{s+\rho} R(z) x^{-s}\|\leq CH^{-\rho}|\im z|^{-4}.
\end{equation}
\end{prop}

\begin{proof}
If $s=0$, the result follows from Corollary~\ref{cor:ext-res-bounds-sc}.
If $n\geq 4$, $0<s<1$, it follows from Lemma~\ref{lemma:res-small-conj}.

Assume
now that $1\leq s<\min(2,n/2)$. If $n=3$, let $\rho'=s/2$; if $n\geq 4$, let
$\rho'=s-1$. In either case, $0\leq \rho'<1$; if $n=3$ then in addition $\rho'>1/2$.
Let $$L'=Lx\rho \in x^\rho \Diffsc^1(X).$$
Write
\begin{gather*}
L'x^{s}R(z) x^{-s}=L' R(z)+L'x^{s}[R(z),x^{-s}]\\
=L' R(z)-\left(L'x^{s}R(z)x^{-\rho'}\right)
\left(x^{\rho'}[H^2(\Delta+V),x^{-s}]R(z)\right).
\end{gather*}
First, by Corollary~\ref{cor:ext-res-bounds-sc} since $\rho'<\min(n/2,2),$
$$
\|L'R(z)\|\leq CH^{-\rho'}|\im z|^{-2}.
$$
Next, we apply \eqref{eq:res-est-8},  with $s=\rho'$ and
$\rho$ replaced by $\rho+s-\rho'$.  The hypotheses are satisfied as
$\rho'+(\rho+s-\rho') = \rho+s<\min(n/2,2)$ and
as $0\leq \rho+s-\rho'<1$ if $n=3$ since $\rho+s<3/2$, $\rho'=s/2$
and $s\geq 1$, resp.\ $0\leq \rho+s-\rho'=\rho+1$ if $n\geq 4$,
$$
\|L'x^{s}R(z)x^{-\rho'}\|\leq CH^{-s-\rho+\rho'}|\im z|^{-2}.
$$
On the other hand,
$x^{\rho'}[\Delta+V,x^{-s}]\in x^{2-s+\rho'}\Diffb^1(X)$, so
by Corollary~\ref{cor:ext-res-bounds}
(taking into account that if $n\geq 4$, $2-s+\rho'=1$, and if
$n=3$ then $1\leq 2-s+\rho'<3/2$),
\begin{gather*}
\|x^{\rho'}[(\Delta+V),x^{-s}]R(z)\|\leq CH^{-2+s-\rho'}|\im z|^{-2}.
\end{gather*}
Combining these results, we have shown that
$$
\|L'x^{s}R(z) x^{-s}\|\leq CH^{-\rho}|\im z|^{-4},
$$
completing the proof if $1\leq s<\min(2,n/2)$. If $n\geq 4$, we already covered
the case of $s<1$ in Lemma~\ref{lemma:res-small-conj}, hence if $n\geq 4$, the proof is complete.

So suppose that $n=3$.
If $0<s<1$, and $0\leq \rho<1/2$, then one can interpolate between $s=0$ and $s=1$
with the same (fixed) value of $\rho$ to obtain \eqref{eq:res-est-up-to-3-halves},
completing the proof if $0\leq \rho<1/2$ in the full range of $s$, $0\leq s<3/2$
(subject to $s+\rho'<3/2$). Finally, with $0\leq s+\rho'<3/2$ fixed, one can interpolate
between $s=0$ and $\rho=0$ to obtain the full result.
\end{proof}

In view of \eqref{eq:func-calc}, Proposition~\ref{prop:res-est-up-to-2}
proves Proposition~\ref{prop:conjugated-cutoff}.

We need a second, much less delicate, property of spectral cutoffs. Note that
here we do not need to work with low energies: the cutoff $\psi$ is fixed, and
there is no $H$ in the statement of the proposition. (If we worked
with $\psi(H^2(\Delta+V))$, $H$ could be traced through the argument given
below to yield polynomially growing bounds in $H$, i.e.\ one can gain decay in $t$
at the cost of losing powers of $H$.)

\begin{prop}\label{prop:localized-spectral-cutoff}
Suppose that $\phi,\chi\in\CI_c(\RR)$ and $\supp(1-\chi)\cap\supp\phi=\emptyset$,
$\psi\in\CI_c(\RR)$. Also let $r=x^{-1}$.
Then for $L\in\Diffsc^1(X)$, $N\in\NN$,
$$
\|L\phi(r/t)\psi(\Delta+V)(1-\chi(r/t))\|_{\cL(L^2,L^2)}\leq C_N t^{-N},\ t\geq 1.
$$
\end{prop}

\begin{proof}
We first remark that
$$
[x^2\pa_x,\phi(1/(xt))]=-t^{-1}\phi'(1/(xt)),
$$
while $xD_{y_j}$ and elements of $\CI(X)$ commute with the multiplication operator
by $\phi(1/(xt))$.
Thus, by an inductive argument, for $Q\in\Diffsc^m(X)$, we have
\begin{equation}\label{eq:wave-space-commutator}
[Q,\phi(1/(xt))]=\sum_{|\alpha|+j\leq m-1}\sum_k f_{j,\alpha,k}(t) \phi_{j,\alpha,k}(1/(xt)) 
a_{j,\alpha,k}(x^2 D_x)^j (xD_y)^\alpha
\end{equation} 
where the sum over $k$ is finite,
$f_{j,\alpha,k}\in S^{-1}([1,\infty))$ (i.e.\ is $\CI$ and is a symbol of
order $-1$ at infinity),
$a_{j,\alpha,k}\in\CI(X)$, $\phi_{j,\alpha,k}\in\CI_c(\RR)$ and $\supp\phi_{j,\alpha,k}
\subset\supp d\phi$.
Equivalently, we may put all of the factors $\phi_{j,\alpha}(1/(xt))$ on the right
(at the cost of changing these factors as well as the other coefficients), so
\begin{equation}\label{eq:wave-space-commutator-right}
[Q,\phi(1/(xt))]=\sum_{|\alpha|+j\leq m-1}\sum_k\tilde f_{j,\alpha,k}(t)
\tilde a_{j,\alpha,k}(x^2 D_x)^j (xD_y)^\alpha  \tilde\phi_{j,\alpha,k}(1/(xt)),
\end{equation}
with the tilded objects having the same properties as the untilded ones above.

For each $t\geq 1$ we may write,
$L\phi(r/t)=\phi(r/t) L+[L,\phi(r/t)]$ and use that for any $Q\in\Diffsc^m(X)$,
$Q\psi(\Delta+V)$ is bounded on $L^2$ (by elliptic regularity), so
an immediate consequence of \eqref{eq:wave-space-commutator}
is that
$L\phi(r/t)\psi(\Delta+V)(1-\chi(r/t))$ is a bounded operator on $L^2$ which
is uniformly bounded in $t\geq 1$;
we now must obtain improved (decaying) bounds in $t$.

To do so, we work with the resolvent $R(z)=(\Delta+V-z)^{-1}$, $\im z\neq 0$,
and note that
if $\rho\in\CI_c(\RR)$ and $1-\chi$ have disjoint support, so $\rho(1-\chi)=0$,
then commuting $\rho$ through the resolvent,
\begin{equation}\begin{split}\label{eq:localizer-expansion}
\rho(1/(xt))R(z)(1-\chi(1/(xt)))&=-[R(z),\rho(1/(xt))](1-\chi(1/(xt)))\\
&=R(z)[\Delta+V,\rho(1/(xt))]R(z)(1-\chi(1/(xt))).
\end{split}\end{equation}
By \eqref{eq:wave-space-commutator-right}, this has the form
\begin{equation}\label{eq:localizer-expansion-detail}
\sum_{|\alpha|+j\leq 1}\sum_k f_{j,\alpha,k}(t) R(z) 
a_{j,\alpha,k}(x^2 D_x)^j (xD_y)^\alpha \rho_{j,\alpha,k}(1/(xt)) R(z) (1-\chi(1/(xt))),
\end{equation}
with $\supp\rho_{j,\alpha,k}\subset\supp d\rho$. Thus,
$\rho_{j,\alpha,k}(1/(xt)) R(z) (1-\chi(1/(xt)))$ is of the same form as the
left hand side of
\eqref{eq:localizer-expansion}, so \eqref{eq:localizer-expansion} can
be further expanded by expanding these last three factors
in \eqref{eq:localizer-expansion-detail} as in
\eqref{eq:localizer-expansion}. Note that as we expand, each time
we obtain a factor like $f_{j,\alpha,k}\in S^{-1}([1,\infty))$; these commute with
all other factors, and after $N$ iterations, the product of these is in $S^{-N}([1,\infty))$.
Inductively, doing $N$ iterations,
we deduce that $\phi(1/(xt))R(z)(1-\chi(1/(xt)))$ is a sum of terms
of the form
\begin{equation}\label{eq:t-loc-exp}
f(t) R(z)A_1 R(z) A_2\ldots R(z) A_N \tilde\phi(1/(xt)) R(z)(1-\chi(1/(xt))),
\end{equation}
with $A_j\in\Diffsc^1(X)$, $\tilde\phi\in\CI_c(\RR)$ with $\supp\tilde\phi\subset
\supp d\phi$, and $f\in S^{-N}([1,\infty))$, and one can also interchange
$A_N$ and $\tilde\phi(1/(xt))$ if convenient (at the cost of obtaining different
operators in the same class).

But $\|A_j R(z)\|\leq C|\im z|^{-1}$ for $z$ in a compact set, and
similarly for $LR(z)$ when $L\in\Diffsc^1(X)$ (see
Corollary~\ref{cor:ext-res-bounds-sc} with $H$ fixed) so using that
$\tilde\phi$ and $1-\chi$ are uniformly bounded in $\sup$ norm, hence
as bounded operators on $L^2$, we deduce that for $L\in\Diffsc^1(X)$,
$$
\|L\phi(1/(xt))R(z)(1-\chi(1/(xt)))\|\leq Ct^{-N}|\im z|^{-N-1}.
$$
The Cauchy-Stokes formula, with $\tilde\psi\in\CI_c(\Cx)$ an almost
analytic extension of $\psi$,
\begin{equation}\begin{split}\label{eq:func-calc-2}
&L\phi(1/(xt))\psi(\Delta+V) (1-\chi(1/(xt)))\\
&=\frac{1}{2\pi}\int\overline{\pa}\tilde\psi(z)
L\phi(1/(xt))R(z) (1-\chi(1/(xt))) \,dz\,d\bar z,
\end{split}\end{equation}
now immediately proves the proposition.
\end{proof}

In fact, it is useful to add a weight in $x$ as well:

\begin{prop}\label{prop:weighted-localized-spectral-cutoff}
Suppose that $\phi,\chi\in\CI_c(\RR)$ and $\supp(1-\chi)\cap\supp\phi=\emptyset$,
$\psi\in\CI_c(\RR)$. Also let $r=x^{-1}$.
Then for $L\in\Diffsc^1(X)$, $N\in\NN$, $m\in\NN$
$$
\|Lr^m\phi(r/t)\psi(\Delta+V)(1-\chi(r/t))r^m\|_{\cL(L^2,L^2)}\leq C_N t^{-N},\ t\geq 1.
$$
\end{prop}

\begin{proof}
First, the $r^m=x^{-m}$ in front of $\psi(\Delta+V)$ is harmless since on the support
of $\phi$ it is bounded by $t^m$, so it can be absorbed into $t^{-N}$ on the
right hand side if we increase $N.$  Moreover, for each $t$,
$Qx^{m}\psi(\Delta+V)x^{-m}$
is bounded for $Q\in\Diffsc^1(X)$,
as is immediate from the scattering calculus of Melrose as
$\psi(\Delta+V)$ is a pseudodifferential operator of order $-\infty$.
However, there is a simple direct argument: via the Cauchy-Stokes formula
this reduces to a boundedness statement for $Qx^{m}R(z)x^{-m}$. This in turn
is shown by
commuting $x^{-m}$ through $R(z)$, much as we commuted $\phi(r/t)$
above, in this case gaining a power of $x$ (instead of $t^{-1}$) each time
we do a commutation, so in $m$ steps we reduce to a product of $m+1$ resolvents
and bounded functions, giving a bound
\begin{equation}\label{eq:weighted-res-bd}
\|Qx^{m}R(z)x^{-m}\|\leq C|\im z|^{-m-1}.
\end{equation}
Thus, as in the previous proposition, we only
need to prove the decaying uniform estimate
in $t$. (Strictly speaking, one has to regularize, replacing $x^{-m}$ by
$(1+\ep x^{-1})^{-m} x^{-m}$, which is bounded for $\ep>0$, and let $\ep\to 0$, in
which case the commutator argument presented above gives uniform control of the terms.)

We thus proceed as above, so by
\eqref{eq:t-loc-exp} we only need to prove that
$$
A_N\tilde\phi(1/(xt)) R(z)(1-\chi(1/(xt)))x^{-m}
$$
is $L^2$ bounded. As in the fixed $t$ setting,
it is convenient to insert a factor of $x^m$ on the left,
using that $x^{-m}\lesssim t^m$ on the support of $\tilde\phi$, and thus,
commuting $\tilde\phi$ through $A_N$, and using the uniform boundedness of
$\tilde\phi$ and $1-\chi$ in the $\sup$ norm, we are reduced to showing that
$A_N' x^m R(z)x^{-m}$
is bounded (with a polynomial bound in $|\im z|^{-1}$).
But this is exactly the content of \eqref{eq:weighted-res-bd}, completing the
proof.
\end{proof}

\section{Global Morawetz estimate}\label{sec:Morawetz}

In this section, we return to the notation $r=1/x$ to facilitate
comparisons with the usual method of Morawetz estimates. To simplify the
computation, we assume
$g_1=0$ with the notation of \eqref{eq:metric-form}--\eqref{eq:g-1-form}. Then,
in the paragraph of \eqref{eq:g-1-reinstate} we reinstate arbitrary $g_1$ (with
$\decay>0$).
Thus in the present notation, we have
\begin{equation}\label{eq:warped-product}
g=g_0= dr^2+r^2 h(1/r,y,dy)
\end{equation}
with $r \in (r_0,\infty).$
Then we compute
\begin{multline}\label{eq:pa-r-adjoint}
\pa_r^* =-\pa_r-\pa_r(\log (r^{n-1} \sqrt h)) =-\pa_r-\frac{n-1}r-\frac
12 \pa_r \log h(1/r)\\ =-\pa_r-\frac{n-1}r+e\in r^{-1}\Diffb^1(X)
\end{multline}
with $e \in S^{-2}(X).$

We as usual may write, in the ends of $X$ where $r \gg 0,$
$$
\Lap_g = \pa_r^*\pa_r + r^{-2} \Lap_Y\in r^{-2}\Diffb^2(X)
$$
where $\Lap_Y$ is the family of Laplacians on the boundary at infinity
defined by the family of metrics $h(1/r,\bullet).$

Now in general, for $F(r) \in \CI([0,\infty))$ we let
\begin{equation}\label{As}A_F= (1/2) (F(r)\pa_r-\pa_r^*F(r)),\end{equation}
which is thus skew-adjoint.
The principal symbol of the commutator
of $\Lap_g$ with $A_F$ is straightforward to compute, and one deduces that
$[\Lap,A_F]$ and $2\pa_r^*F'(r)\pa_r+2r^{-3}F(r)\Lap_Y$ have the same
principal symbol, so their difference is first order.
Since they are both formally self-adjoint and real, the same holds for the
difference, which is thus zero'th order, and can be computed by applying
it to the constant function $1$:
\begin{equation*}\begin{aligned}
\left([\Lap,A_F]-(2\pa_r^*F'(r)\pa_r+2r^{-3}F(r)\Lap_Y)\right)1&=\Lap_g A_F 1\\
&=-\frac 12 \Lap_g \pa_r^* (F(r))\\
\end{aligned}\end{equation*}

We note some particular instances of this computation.  Let
$\cutoff(r)$ be a function on $\RR$ such that
$$
\cutoff (r)=\begin{cases}
0 & r\leq 1/2\\
1 & r \geq 4,
\end{cases}
$$
and
$$
\cutoff'\geq 0,\quad \cutoff'(r)\geq 1/4 \text{ for } 1\leq r \leq 2.
$$
Let $\Lambda \gg 0$ and $\varkappa$ be a small constant.
We compute the following as zero'th order terms for these various cases:
$$\begin{aligned}
F(r) =r^{s} &\Longrightarrow
-\frac 12 \Lap_g \pa_r^*
(F(r))=\frac{n-1}{2}\,(s-1)(n+s-3)r^{s-3}+O(r^{s-4}),\\
F(r) =(1-1/\log r) &\Longrightarrow
-\frac 12 \Lap_g \pa_r^* (F(r))=(1+O(1/\log r))\frac{(n-1)(n-3)}{2r^{3}}\\
F(r) =\varkappa \,\cutoff(r/\Lambda) &\Longrightarrow -\frac 12 \Lap_g \pa_r^*
(F(r))=\varkappa \,O(r^{-3})
\end{aligned}$$
(with the ``big-Oh'' term in the last case being uniformly bounded as
$\Lambda \to \infty$).

Thus in particular, we may let $F(r)=r^s$ and (abusing notation) denote
$A_F$ by $A_s$ in this case.  Then
\begin{equation}\label{comm1}
[\Lap_g+V, A_s] 
= 2s\pa_r^*r^{s-1}\pa_r+2r^{s-3}\Lap_Y-\frac{n-1}{2}\,(s-1)(n+s-3)r^{s-3}
+ r^{s-4}Q_Y+f,
\end{equation}
where $Q_Y$ denotes an operator of second order involving only
derivatives in $Y$ with coefficients in $S^0(X)$ and $f \in S^{s-3-\decay}(X).$
The two principal terms have the same (non-negative!) sign if $s\geq 0$, with
the sign being definite if $s>0$; if $n\geq 3$,
the zero'th order term has the same sign if in addition $s\leq 1$; the sign
is definite if $s<1$, and in case $n=3$, $s>0$.

On the other hand, in dealing with radial derivatives we will employ $A_F$
with $F(r) = (1-1/\log r).$  We let $\pA$ denote the operator in this case,
and compute
$$
\pA = (1/2) \bigg( \big( 1-\frac{1}{\log r} \big) A+ A\big(
1-\frac{1}{\log r} \big) \bigg),
$$
where
$$
A \equiv A_0  = \frac 12 \big( \pa_r -\pa_r^*\big).
$$
Then we have
\begin{equation}\label{comm-a}
[\Lap_g+V, \pA] 
= 2\pa_r^*\frac{1}{r(\log r)^{2}}\pa_r+(1+e_1)2r^{-3}\Lap_Y+(1+e_2)\frac{(n-1)(n-3)}{2r^{3}}
+ r^{-4}Q_Y,
\end{equation}
with $e_i =O((\log r)^{-1})$ a symbol in $S^0(X).$
(We have recycled the notation $Q_Y$ to denote different operators of the same form.)

Finally, in obtaining $\ell^\infty$--$\ell^1$ bounds, we will employ
$$
F(r) = 1- \frac 1 {\log r} + \varkappa \,\cutoff\big(\frac r\Lambda\big);
$$
with $\varkappa$ a (small) positive constant.  Let $\ppA$ denote the corresponding operator.  Thus
\begin{equation}\begin{split}\label{comm-ppa}
[\Lap_g+V, \ppA] 
= &2\pa_r^*\big( \frac{\varkappa}{\Lambda} \cutoff'(r/\Lambda)+ \frac{1}{r(\log
  r)^{2}}  \big)\pa_r\\
&\qquad+(1+e_1)2r^{-3}\Lap_Y+(1+e_2)\frac{(n-1)(n-3)}{2r^{3}}
+ r^{-4}Q_Y,
\end{split}\end{equation}
with $e_1 =O((\log r)^{-1})$ a symbol in $S^0(X)$ and $e_2 =O((\log
r)^{-1})+\varkappa\, O(1)$ a symbol in $S^0(X),$ bounded by $1/2$ when $r \gg
0$ and $\varkappa \ll 1.$

We now return to arbitrary metric $g=g_0+g_1$, with $g_0$ of the warped
product form \eqref{eq:warped-product}. We write $\Delta_{g_0}$ for the Laplacian
of $g_0$ and $B^{*,0}$ for the adjoint of an operator $B$
with respect to $g_0$, while $B^*$ is written for the adjoint with respect to $g$.
Thus, in \eqref{eq:pa-r-adjoint}--\eqref{comm-a}, all adjoints are of the ${}^{*,0}$ type,
and all Laplacians are that of $g_0$. Since $A_s$ and $\pA$ depend on the
metric via the adjoints, we write $A_{s,0}$ and $\pA_0$ for the $g_0$ versions.
Now, as
\begin{equation}\label{eq:g-1-reinstate}
g_1\in S^{-\decay}(X;\Tsc^*X\otimes\Tsc^*X)=S^{2-\decay}(X;\Tb^*X\otimes\Tb^*X),
\end{equation}
we have
$$
\sqrt{g}=\sqrt{g_0}(1+\tilde g),\ \tilde g\in S^{-\decay}(X),
$$
and thus
$$
\pa_r^*-\pa_r^{*,0}\in S^{-1-\decay}(X),
$$
i.e.\ the effect on \eqref{eq:pa-r-adjoint} is that $e$ is replaced by a slightly
less decaying symbol. Correspondingly,
$$
A_s-A_{s,0}\in S^{s-1-\decay}\Diffb^1(X),\ \pA-\pA_0\in S^{-1-\decay}\Diffb^1(X),\ \ppA-\ppA_0\in S^{-1-\decay}\Diffb^1(X).
$$
Also,
$$
\Delta_g-\Delta_{g_0}\in S^{-2-\decay}\Diffb^2(X).
$$
Combining these, and expanding the commutators,
\begin{equation}\begin{aligned}\label{eq:perturbed-comm}
&[\Delta_g+V,A_s]-[\Delta_{g_0}+V,A_{s,0}]\in S^{-3+s-\decay}\Diffb^2(X),\\
&[\Delta_g+V,\pA]-[\Delta_{g_0}+V,\pA_0]\in S^{-3-\decay}\Diffb^2(X)\\
&[\Delta_g+V,\ppA]-[\Delta_{g_0}+V,\ppA_0]\in S^{-3-\decay}\Diffb^2(X),
\end{aligned}\end{equation}
as always with uniform estimates as $\Lambda \to \infty$ in the $\ppA$ case.

In order to make the foregoing into a global computation, we need to add a cutoff that
localizes near infinity.
So let $\chi_0(r)$ be a cutoff function equal
to $0$ for $r<1$ and $1$ for $r>2.$  Let $\chi(r)=\chi_0(r/R)$ with $R
\gg 0.$  Then
\begin{equation}\label{comm2}
\begin{aligned}
\mbox{}[\Lap_g+V, \chi(r)A_s]
&= 2s\pa_r^*\chi(r) r^{s-1}\pa_r+2r^{s-3}\chi(r)\Lap_Y+\frac{n-1}{2}\,(1-s)(n+s-3)r^{s-3}
\\ &+ \nabla_X^* (1-\chi(r)) \nabla_X + r^{s-4}E_s+r^{s-3-\decay}\tilde E_s,
\end{aligned}
\end{equation}
where $E_s\in \Diffsc^2(X)$, $\tilde E_s\in\Diffb^2(X)$,
and
\begin{equation}\label{comm-b}
\begin{aligned}
\mbox{}[\Lap_g+V, \chi(r)\pA ]
&= 2\pa_r^*\frac{\chi(r)}{r(\log r)^{2}}\pa_r+(1+e_3) 2r^{-3}\chi(r)\Lap_Y+(1+e_4)\frac{(n-1)(n-3)}{2r^{3}}
\\ &+\nabla_X^* (1-\chi(r)) \nabla_X + r^{-4}E_0+r^{-3-\decay}\tilde E_0
\end{aligned}
\end{equation}
where $E_0\in \Diffsc^2(X)$, $\tilde E_0\in\Diffb^2(X)$
(and are different from the $E_s$, $\tilde E_s$ in
\eqref{comm2}), and $e_i = O((\log r)^{-1})$ satisfy
symbol estimates.  Since on any compact set, the terms with $e_i$ can be
absorbed into $E_s$, we may assume here that $|e_i|<1/2$ so that
$$
1+e_i>1/2.
$$
For the same reason, by shifting a large compactly supported part of
the $\tilde E_s$ term in \eqref{comm2} into the faster-decaying $E_s$
we may assume that $\tilde E_s$ can be absorbed into the first three
terms of \eqref{comm2} without compromising their positivity: for any
$\delta>0$ we may assume
\begin{equation}\begin{split}\label{eq:absorb-tilde-Es}
&|\langle r^{s-3-\decay}\tilde E_s v,v\rangle|\\
&\qquad\leq
\delta
\Big(2s\|\sqrt{\chi(r)} r^{(s-1)/2}\pa_r v\|^2+2\|r^{(s-3)/2}\sqrt{\chi(r)}\nabla_Y v\|^2\\
&\qquad\qquad\qquad\qquad+\frac{n-1}{2}\,(1-s)(n+s-3)\|r^{(s-3)/2} v\|^2\Big)
\end{split}\end{equation}
for all $v\in\dCI(X)$ (hence by density whenever the right hand side is finite),
with.  We will take $\delta=1/2$.  Likewise, we may assume that the terms in
\eqref{comm-b} satisfy:
\begin{equation}\begin{split}\label{eq:absorb-tilde-Es-b}
&|\langle r^{-3-\decay}\tilde E_0 v,v\rangle|\\
&\qquad\leq
\delta 
\Big(2\|\frac{\sqrt{\chi(r)}}{r^{1/2}(\log r)}\pa_rv\|^2
+2\|r^{-3/2}\sqrt{\chi(r)}\nabla_Y v\|^2+\frac{(n-1)(n-3)}{2}\|r^{-3/2} v\|^2\Big).
\end{split}\end{equation}
We will take $\delta=1/4$ here.

Exactly the same computation applies in the case of $\ppA,$ but with the
addition of an extra $\pa_r^2$ term:
\begin{equation}\label{comm-bb}
\begin{aligned}
\mbox{}[\Lap_g+V, \chi(r)\ppA ]
= &2\pa_r^*\big( \frac{\chi(r)}{r(\log r)^{2}} +
\frac{\varkappa}{\Lambda} \cutoff'(r/\Lambda)  \big)\pa_r+(1+e_3)
2r^{-3}\chi(r)\Lap_Y\\
&+(1+e_4)\frac{(n-1)(n-3)}{2r^{3}}
\\ &+\nabla_X^* (1-\chi(r)) \nabla_X + r^{-4}E_0+r^{-3-\decay}\tilde E_0,
\end{aligned}
\end{equation}
with error terms estimated as above (provided $\varkappa$ is taken
sufficiently small).

We now consider the pairing
\begin{equation}\label{eq:pairing-1}
\int_0^T \ang{\chi(r) \pA (\Box+V) u,u} \, dt.
\end{equation}
While we will manipulate this expression for an arbitrary solution
$u,$ we remark that the calculations in the following paragraph are unchanged
if we replace $u$ by
$$
\Psi_H u \equiv \psi(H^2(\Lap_g+V)) u,
$$
which indeed we will do below.

We first remark that 
moving $\chi(r)$ and $\pA$ to the right slot of the pairing \eqref{eq:pairing-1} and inserting
the weight $r^{1/2}(\log r)$, resp.\ its reciprocal, in the two slots,
we obtain, for $\gamma>0$,
\begin{equation}\begin{split}\label{eq:comm-b-PDE}
&\left|\int_0^T \ang{\chi(r) \pA (\Box+V) u,u} \, dt\right|\\
&\leq \left(\int_0^T\|r^{1/2}(\log r)(\Box+V)u\|^2\,dt\right)^{1/2}
\left(\int_0^T\|r^{-1/2}(\log r)^{-1}(\pA)^*\chi(r)u\|^2\,dt\right)^{1/2}\\
&\leq \gamma^{-1}\int_0^T\|r^{1/2}(\log r)(\Box+V)u\|^2\,dt
+\gamma\int_0^T\|r^{-1/2}(\log r)^{-1}(\pA)^*\chi(r)u\|^2\,dt,
\end{split}\end{equation}
where $\ang{\cdot,\cdot}$ denotes the spatial inner product.  On the
other hand, integrating \eqref{eq:pairing-1} twice by parts yields
\begin{equation}\label{aaa}
-\ang{\chi(r) \pA u_t,u}|_0^T + \ang{\chi(r) \pA u,
  u_t}_0^T-\int_0^T \ang{ [\chi(r)\pA, \Lap_g+V] u, u} \, dt.
\end{equation}
Let $$\E(s) = \big(\norm{\nabla_g u}^2 + \norm{\pa_t u}^2+\norm{\sqrt V u}^2\big)|_{t=s}$$ denote
the energy norm at fixed time.  By the Hardy/Poincar\'e inequality, we have
$$
\norm{\chi \pA u(t)}^2 \lesssim \E(t).
$$
We further compute, as usual, that
$$
\E'(t) =2 \Re \ang{u_t, (\Box+V) u} \leq  2\E(t)^{1/2}
\norm{(\Box+V) u}.
$$
Hence
$$
\begin{aligned}
\E(T)^{1/2} &\leq \E(0)^{1/2} + \int_0^T \norm{(\Box +V) u} \, dt \\
&\lesssim \E(0)^{1/2}
+ \int_0^T \norm{(\Box +V) u} \, dt
\end{aligned}
$$
Consequently, the first two terms in \eqref{aaa} are bounded by a
multiple of
$$
\E(0)+ \left(\int_0^T \norm{(\Box +V) u} \, dt\right)^2
$$
uniformly in $t.$\footnote{If preferred, we could of course replace the norm on
the inhomogeneity by the weighted $L^2$ spacetime norm $$\int_0^T\norm{(2+t^2)^{1/4}
\log (2+t^2) (\Box +V) u}^2 \,dt.$$}

Thus, \eqref{comm-b}, \eqref{eq:absorb-tilde-Es-b} and
\eqref{eq:comm-b-PDE} now allow us to estimate (absorbing
lower-order terms into positive ones, and choosing $\gamma$ small so that
the second term on the right hand side of \eqref{eq:comm-b-PDE} can be
absorbed into the left hand side below):
\begin{equation}\label{foo}
\begin{aligned}
\int_0^T &\ang{\pa_r^* r^{-1} \chi(r) \log(r)^{-2}\pa_r u, u}+ \ang{r^{-3}u,u}\\
&\qquad\qquad+\ang{r^{-3}\chi(r)\Lap_Yu,u} + \ang{\nabla_X^* \chi(r) \nabla_X u,
  u}\, dt \\  &\lesssim  \int_0^T
\ang{r^{-4}E_0u,u} \, dt + \E(0)+\big(\int_0^T\norm{ (\Box +V) u}\,dt\big)^2.
\end{aligned}
\end{equation}

We apply this estimate to spectrally localized data, i.e.\ replace $u$ by $\Psi_H u$,
as indicated at the beginning of the previous paragraph.
We state the following lemma more generally than is immediately necessary, in the $n\geq 3$ setting.
\begin{lemma}\label{lemma1}
Suppose $n\geq 3$.
If $E \in \Diffsc^2(X)$ and $u$ is as above, then for $0\leq\rho$, $0\leq s$,
$0\leq \rho+s<\min(2,n/2)$,
$$
\abs{\ang{r^{-2(s+\rho)}E \Psi_Hu,\Psi_Hu}} \lesssim H^{-2\rho}\norm{r^{-s} \Psi_H u}^2.
$$
\end{lemma}
\begin{proof}
It suffices, by rearranging the LHS (as commutator
terms require the same form of estimates) to show that for $L \in \Diffsc^1(X),$
$$
\norm{r^{-(s+\rho)} L \Psi_Hu} \lesssim H^{-\rho} \norm{r^{-s} \Psi_Hu}.
$$
Let $\phi\in\CI_c(\RR)$ be identically $1$ on $\supp\psi$.
By Proposition~\ref{prop:conjugated-cutoff},
$$
H^{\rho} r^{-(s+\rho)} L \phi(H^2(\Lap_g+V)) r^{s}
$$
is uniformly bounded in $H$, so for all $v$,
$$
\norm{H^{\rho} r^{-(s+\rho)} L \phi(H^2(\Delta+V)) v} \lesssim  \norm{r^{-s} v}.
$$
Applying this with $v=\Psi_H u$ completes the proof of the lemma.
\end{proof}

Taking $s=3/2$, $\rho=(1-\ep)/2,$ $0<\ep<1$,
which we may if $n\geq 4$, we deduce:
\begin{lemma}\label{cor-of-lemma1}
If $n\geq 4$, $E \in \Diffsc^2(X)$ and $u$ is as above, then for $0<\ep<1,$
$$
\abs{\ang{r^{-4+\ep}E \Psi_Hu,\Psi_Hu}} \lesssim H^{-1+\ep}\norm{r^{-3/2} \Psi_H u}^2.
$$
\end{lemma}

As a consequence, we can estimate, a fortiori,
$$
\ang{r^{-4}E_0 \Psi_Hu,\Psi_Hu} \lesssim H^{-1+\ep} \norm{r^{-3/2} \Psi_Hu}^2
$$
so that for $H$ sufficiently large, this term may be absorbed in the
main term on the left of \eqref{foo}.  Thus, for $H$ sufficiently large,
$$
\begin{aligned}
\int_0^T &\norm{r^{-1/2} (\log r)^{-1} \sqrt\chi(r) \pa_r \Psi_Hu}^2 +
\norm{r^{-3/2} \Psi_Hu}^2\\
&\qquad\qquad+\norm{r^{-3/2} \sqrt\chi(r)\nabla_Y \Psi_Hu}^2 +
\norm{\sqrt{1-\chi} \nabla_X \Psi_Hu}^2 \,  dt\\ &\lesssim 
\E(0)+
\big(\int_0^T\norm{ (\Box +V) u}\,dt\big)^2,
\end{aligned}
$$
with constants independent of $T.$
This concludes the proof of Theorem~\ref{thm1}.\qed

To prove Theorem~1.1' we proceed similarly, but with $\ppA$ replacing
$\pA.$ We now let
\begin{equation}\label{eq:Ups-def}
\Upsilon_k=\{r \in [2^k, 2^{k+1}]\}
\end{equation}
denote a dyadic decomposition in radial shells.  Then we obtain by the
same argument
\begin{equation}\label{eq:comm-bprime-PDE}\begin{aligned}
&\left|\int_0^T \ang{\chi(r) \ppA (\Box+V) u,u} \, dt\right|
\\&\quad=\left|\int_0^T \ang{r^{1/2} (\Box+V) u,r^{-1/2}(\ppA)^*
\chi(r) u} \, dt\right|\\
&\quad\leq \int_0^T \sum_k \int_{\Upsilon_k}  \left\lvert r^{1/2} (\Box+V)
  u\right \rvert \cdot \left \lvert r^{-1/2}(\ppA)^*
 \chi(r) u \right \rvert \,dg\, dt\\
&\quad\leq \sum_k \norm{ r^{1/2} (\Box+V) u}_{L^2([0,T] \times \Upsilon_k)}\norm{r^{-1/2}(\ppA)^*
 \chi(r) u }_{L^2([0,T] \times \Upsilon_k)}\\
&\quad \leq \norm{ r^{1/2} (\Box+V) u}_{\ell^1(\NN; L^2([0,T] \times \Upsilon_k))}\norm{r^{-1/2}(\ppA)^*
 \chi(r) u }_{\ell^\infty(\NN; L^2([0,T] \times \Upsilon_k))}\\
&\quad \leq \gamma^{-1} \norm{ r^{1/2} (\Box+V) u}_{\ell^1(\NN; L^2([0,T] \times
  \Upsilon_k))}^2 +\gamma\norm{r^{-1/2}(\ppA)^*
 \chi(r) u }_{\ell^\infty(\NN; L^2([0,T] \times \Upsilon_k))}^2
\end{aligned}\end{equation}
Thus, for $\gamma>0$ we have
\begin{equation}\label{footwo}
\begin{aligned}
\int_0^T &\Big(\ang{\pa_r^* r^{-1} \chi(r) \log(r)^{-2}\pa_r u, u}+ \ang{\pa_r^*
  \frac{\varkappa}{\Lambda} \cutoff'(r/\Lambda) \chi(r) \pa_r  u, u}
+ \ang{r^{-3}u,u}\\
&\qquad\qquad+\ang{r^{-3}\chi(r)\Lap_Yu,u} + \ang{\nabla_X^* \chi(r) \nabla_X u,
  u}\Big)\, dt \\  &\lesssim  \int_0^T
\ang{r^{-4}E_0u,u} \, dt + \E(0)+\left(\int_0^T\norm{ (\Box +V) u}\,dt\right)^2
\\ &\qquad\qquad+\gamma^{-1} \norm{ r^{1/2} (\Box+V) u}_{\ell^1(\NN; L^2([0,T] \times
  \Upsilon_k))}^2\\
&\qquad\qquad+\gamma\norm{r^{-1/2}(\ppA)^*
 \chi(r) u }_{\ell^\infty(\NN; L^2([0,T] \times \Upsilon_k))}^2.
\end{aligned}
\end{equation}
Now we take the supremum of the LHS over $\Lambda \in \{2^k, k \in \NN\}$
to obtain
\begin{equation}\label{foothree}
\begin{aligned}
\int_0^T &\Big(\ang{\pa_r^* r^{-1} \chi(r) \log(r)^{-2}\pa_r u, u}
+ \ang{r^{-3}u,u}+\ang{r^{-3}\chi(r)\Lap_Yu,u}\\
&\qquad\qquad\qquad+ \ang{\nabla_X^* \chi(r) \nabla_X u,
  u}\Big)\, dt
\\
&\qquad\qquad+\norm{\chi(r) r^{-1/2}\pa_r
 u }_{\ell^\infty(\NN; L^2([0,T] \times \Upsilon_k))}^2 \\  &\lesssim  \int_0^T
\ang{r^{-4}E_0u,u} \, dt + \E(0)+\left(\int_0^T\norm{ (\Box +V) u}\,dt\right)^2
\\ &\qquad\qquad+\gamma^{-1} \norm{ r^{1/2} (\Box+V) u}_{\ell^1(\NN; L^2([0,T] \times
  \Upsilon_k))}^2\\
&\qquad\qquad+\gamma\norm{r^{-1/2}(\ppA)^*
 \chi(r) u }_{\ell^\infty(\NN; L^2([0,T] \times \Upsilon_k))}^2.
\end{aligned}
\end{equation}
where we have of course used the fact that $r \sim \Lambda$ on $\supp
\cutoff'(r/\Lambda).$  Taking $\gamma$ sufficiently small to absorb the
last term in the LHS yields and applying the resulting estimate to
spectrally localized data as above yields Theorem~1.1'.\qed

\section{A local Morawetz estimate}\label{sec:local}

In this section, we prove Theorem~\ref{thm2}. We fix some $\tilde\psi\in\CI_c(\RR)$
which is identically $1$ on $\supp\psi$. We write $u$ simply in place
of
$$
\Psi_H u=\psi(H^2(\Delta+V))u
$$
throughout this section to simplify the notation.

As our estimates will be in spacetime, we also define the manifold
with corners given by its compactification,
$$
M=\overline{\RR}\times X,
$$
with $\overline{\RR}$ denoting the compactification of $\RR$ to an
interval and $X$ our scattering manifold, regarded as usual as a
manifold with boundary endowed with the singular metric
\eqref{eq:metric-form}.  In this section we will employ the notation
$\norm{\bullet}_M$ to denote the spacetime $L^2$ norm, and for the
sake of emphasis, we will use $\norm{\bullet}_X$ for the spatial norm,
previously denoted simply $\norm{\bullet}.$  We will likewise let $\nabla_X$
denote the gradient in the spatial directions only, previous denoted
$\nabla_g,$ while letting $\nabla$ denote the spacetime gradient.

We may assume that the  $\delta$ in \eqref{omegadef} is sufficiently small
for convenience (as the result is stronger then); e.g.\ take $\delta<1/4$.
Let $\phi_0\in\CI_c(\RR)$, supported in $[-2,2]$, identically $1$
in $[-1,1]$. With $0<c<3\delta<1$ 
to be fixed, let $\phi(\dummyvar)=\phi_0(\dummyvar/c)$, and let $\tilde\phi(\dummyvar)=
\phi(\dummyvar/(3\delta))$. Also, we let $\chi$ be as in Section~\ref{sec:Morawetz},
see the definition just above \eqref{comm2}, so $\chi\equiv 1$ near infinity.
Now, with
$$
k=2\kappa-1,\ s=1-2\sigma,
$$
let
$$
B_{k,s}=\frac12\left(\phi(r/t)^2\tilde\phi(t)^2 (r/t)^s t^{-k} \chi(r) \pa_r
-\pa_r^*\phi(r/t)^2\tilde\phi(t)^2 (r/t)^s t^{-k}\chi(r)\right),
$$
so on the complement of $\supp (1-\phi\tilde\phi)$
(where $\phi\tilde\phi\equiv 1$) we actually have
$$
B_{k,s}=t^{-k-s}A_s
$$
with $A_s$ given by \eqref{As}.
Note that $\supp d\tilde\phi\cap\supp\phi$ is a compact subset of $M^\circ$,
so we can effectively ignore terms in which $\tilde\phi$ is commuted through
differential operators below.

Before proceeding, we recall that in $\{r>r_0,\ t>1\}\subset M$,
$\Vb(M)$ is spanned by $r\pa_r$, $t\pa_t$
and vector fields on $Y$, over $\CI(M)$. We now blow up the corner $\pa_2 M$, where both
$x=0$ and $t^{-1}=0$, to obtain the manifold
$$
\tilde M=[M;\pa_2 M].
$$
The blowup procedure (cf.\ \cite[Appendix]{RBMSpec}; the simple setting
of \cite[Section~4.1]{Melrose:Atiyah} where a codimension 2 corner is blown up
is analogous to the present case, with our $t^{-1}$ taking the place of $x'$ in
\cite {Melrose:Atiyah})
serves to replace the
corner $t^{-1}=r^{-1}=0$ of $M$ by its inward-pointing spherical
normal bundle, in this case diffeomorphic to $Y$ times an interval.
This new boundary hypersurface will be denoted\footnote{The notation
  is short for ``main face,'' as for $X$ Euclidean (diffeomorphically, not metrically),
this can be identified with an open dense
subset of the boundary of the radial compactification of Minkowski space. Note that
the light cone, $r/t=1$, hits the boundary in the interior of $\mf$. If $r/t=1$ in $\mf$
were blown up inside $\tilde M$, the resulting front face is where
Friedlander's radiation field \cite{Friedlander} can be defined by
rescaling a solution to the wave equation. Indeed, in the interior of this front
face (in $t>0$)
$r-t=(r/t-1)/t^{-1}$ becomes a smooth variable
along the fibers of the front face.}
 $\mf.$  We let $\tf$ denote the ``temporal infinity'' face given by
 the lift of the set $t^{-1}=0$ in $M$ to $\tilde{M}.$
The main consequence of the blowup procedure is that it
introduces $r/t$ as a smooth function where it is bounded, and its differential
is non-vanishing on the interior of $\mf$ (where it is thus one of
the standard coordinates, namely the variable along the interval referred to above;
others being $t^{-1}$ and $y,$ coordinates
on $Y=\pa X$;
we use $t^{-1}$ as a boundary defining function for $\mf$).
However,
$\Vb(\tilde M)$ is still spanned by the lift of the same vector fields, except that
the smooth structure is replaced by $\CI(\tilde M)$. Correspondingly, the
symbol spaces are unaffected by the blow-up. Now, the support of $d\phi$
only intersects the front face of $\tilde M$ among all boundary faces, and
$r/t$ is bounded from both above and below by positive constants
there.
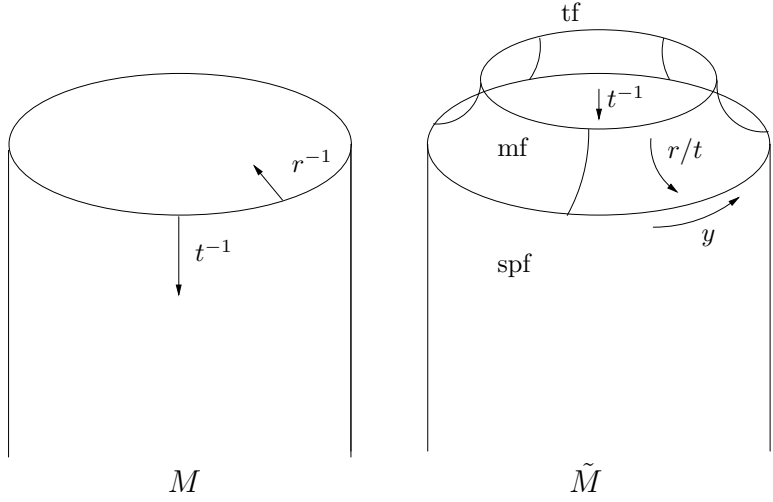
\begin{figure}
\setlength{\unitlength}{0.00041667in}
\begingroup\makeatletter\ifx\SetFigFont\undefined%
\gdef\SetFigFont#1#2#3#4#5{%
  \reset@font\fontsize{#1}{#2pt}%
  \fontfamily{#3}\fontseries{#4}\fontshape{#5}%
  \selectfont}%
\fi\endgroup%
{\renewcommand{\dashlinestretch}{30}
\begin{picture}(9699,6146)(0,-10)
\put(7137,0){\Large{$\tilde{M}$}}
\put(6053.296,5672.395){\arc{1432.185}{6.1394}{6.9106}}
\put(9602.350,5169.252){\arc{1173.430}{1.4621}{3.2487}}
\put(5201.639,4576.997){\arc{4369.283}{6.2672}{6.7887}}
\put(5459.157,5239.166){\arc{1110.177}{0.0670}{1.6540}}
\put(8897.621,4431.604){\arc{1452.961}{2.1434}{3.2304}}
\blacken\path(8390.550,3870.286)(8504.000,3821.000)(8427.090,3917.876)(8390.550,3870.286)
\put(8209.802,5096.991){\arc{3452.002}{0.8990}{1.5742}}
\blacken\path(9205.375,3650.511)(9284.000,3746.000)(9169.689,3698.745)(9205.375,3650.511)
\put(2195,4425){\ellipse{4350}{1800}}
\put(7512,4425){\ellipse{4350}{1800}}
\put(7512,5250){\ellipse{3000}{1260}}
\path(12,4350)(12,450)
\path(4362,4350)(4362,450)
\path(4362,4425)(4362,525)
\path(5337,4425)(5337,525)
\path(9687,4425)(9687,525)
\path(2174,3506)(2174,2516)
\blacken\path(2144.000,2636.000)(2174.000,2516.000)(2204.000,2636.000)(2144.000,2636.000)
\path(3494,3716)(3134,4151)
\blacken\path(3233.620,4077.680)(3134.000,4151.000)(3187.396,4039.426)(3233.620,4077.680)
\path(7514,5126)(7514,4751)
\blacken\path(7484.000,4871.000)(7514.000,4751.000)(7544.000,4871.000)(7484.000,4871.000)
\put(7034,5996){$\tf$}
\put(6224,4271){$\mf$}
\put(6224,2800){$\spf$}
\put(5639,3071){$\sf$}
\put(8819,3191){$y$}
\put(8384,4256){$r/t$}
\put(7619,4896){$t^{-1}$}
\put(3614,4076){$r^{-1}$}
\put(2384,2936){$t^{-1}$}
\put(2037,0){\Large{$M$}}
\put(9043.244,5608.035){\arc{1451.424}{2.6256}{3.3737}}
\end{picture}
}
\caption{The compactified spacetime $M$ and the blown-up
    spacetime $\tilde M.$}
\end{figure}
Thus,
on the support of $d\phi$,
$t^{-1}\Vb(\tilde M)$ is locally generated by $\pa_r$, $\pa_t$ and
$r^{-1}\pa_Y$, i.e.\ by vector fields corresponding to the energy space.
Thus, on $\supp d\phi,$ $\nabla v\in t^\kappa L^2(M)$ implies
$v\in t^{\kappa+1} L^2(M)$, $k<(n-1)/2$,
by a Hardy inequality, and so we obtain more generally:
\begin{lemma}\label{lemma:dphihardy}
Let $\Upsilon$ be an open set in $\tilde M$ with $\overline{\Upsilon}
\cap \pa\tilde M \subset \mf^\circ.$
If $\supp v \subset \Upsilon,$
$$
\nabla v \in t^\kappa L^2(M) \Longrightarrow v\in t^{\kappa+1}\Hb^1(\tilde M).
$$
\end{lemma}
Here the Sobolev space is still relative to the metric density (so
$L^2(M)$ and $L^2(\tilde M)$ are the same, as only the smooth
structure at infinity changed).

Before proceeding, we caution the reader that {\em near temporal infinity}, $\tf$,
i.e.\ the lift of $t=+\infty$ on $M$ to $\tilde M$,
weighted versions of $\Vb(\tilde M)$
do not give rise to the finite length vector fields relative to $dt^2+g$ (i.e.\ the
energy space). Indeed, in the interior of $\tf$, one is in a compact set of
the spatial slice $X$, and $\Vb(\tilde M)$ is spanned
by $t\pa_t$ and smooth vector fields on $X$, while the energy space corresponds
to $\pa_t$ and smooth vector fields on $X$. In order to emphasize the structure
when we stay away from $\tf$ and also from spatial
infinity $\spf$, i.e.\ the lift
of $\pa X\times\overline{\RR}$, we write
$$
\tilde M'=\tilde M^\circ\cup\mf^{\circ}=\tilde M\setminus(\tf\cup\spf).
$$

Our convention below is that inner products and norms are on $M$ unless otherwise
specified by a subscript $X$ (in which case they are on $X$). The inner product
on $M$ on functions is given by integration against the density of
$dt^2-g$, or equivalently that of $dt^2+g$.
The inner product on vectors on $M$ uses the positive
definite inner product, based on $dt^2+g$, unless otherwise specified.
The one case where we take advantage of the ``otherwise specified'' disclaimer
is when we consider the indefinite Dirichlet form in \eqref{eq:D-form-1}, where
we use the indefinite form induced by $dt^2-g$ (so the density is positive
definite, but not the pointwise pairing between vectors) in order to use the PDE.

Now,
\begin{equation}\label{eq:t-weight-comm}
[\pa_t^2,t^{-k-s}]=-2(k+s)t^{-k-s-1}\pa_t+(k+s)(k+s+1)t^{-k-s-2}
\in t^{-k-s-2}\Diffb^1(\overline{\RR}_t).
\end{equation}
Moreover, on $\tilde M'$ (hence on $\supp d\phi$),
$$
B_{k,s}\in t^{-k-1}\Diffb^1(\tilde M'),\ \Box\in t^{-2}\Diffb^2(\tilde M')
\Longrightarrow [\Box,B_{k,s}]\in t^{-k-3}\Diffb^2(\tilde M').
$$
In particular, from \eqref{comm2},
\begin{equation}\begin{split}\label{comm3}
&[\pa_t^2+\Lap_g+V, B_{k,s}]\\
&= t^{-k-s}\phi(r/t)^2\tilde\phi(t)^2\Big(\sum_i Q_i^* r^s t^{-1}R_i\\
&\qquad\qquad+(2s\pa_r^*r^{s-1}\chi(r)\pa_r+2r^{s-3}\chi(r)\Lap_Y+\frac{n-1}{2}\,(1-s)(n+s-3)r^{s-3}\\
&\qquad\qquad\qquad\qquad
+\nabla_X^*(1-\chi(r))\nabla_X+ r^{s-4}E_s+r^{s-3-\decay}\tilde E_s)\Big)+F_{k,s},
\end{split}\end{equation}
where $Q_i\in t^{-1}\Diffb^1(\overline{\RR}_t)$, $R_i\in
r^{-1}\Diffb^1(X)$ contains no $y$-derivatives and is supported on $\supp\chi$,
$E_s\in \Diffsc^2(X)$, $\tilde E_s\in\Diffb^2(X)$,
$F_{k,s}\in t^{-k-3}\Diffb^2(\tilde M)$ and $F_{k,s}$
is supported on
$\supp d\phi$.
Note that the term with $\nabla_X$ has spatially compact support, and could
be absorbed into $E_s$, but writing the commutator in the stated manner
is convenient for it gives a positive definite result modulo $E_s$, which we
later control separately, by low energy techniques. Moreover, for $0<s<1$
(which is the case if $0<\sigma<1/2$), we may assume that the $\tilde E_s$ term
can be absorbed into the first three terms on the right hand side by
\eqref{eq:absorb-tilde-Es}.
Thus, for $0<s<1$
the spatial term, arising from
\eqref{comm2}, is positive modulo the $E_s$ term, while the terms $Q_i^* r^s t^{-1}R_i$
arising from the commutator with $\pa_t^2,$ are indefinite. However,
we can estimate
\begin{equation*}\begin{split}
&\|t^{-(k+s)/2} r^{(s-1)/2}\tilde\phi\phi R_i u\|_M\\
&\qquad\leq C_0 (\| t^{-(k+s)/2} r^{(s-1)/2}\tilde\phi\phi\pa_r u\|_M+\|t^{-(k+s)/2} \tilde\phi\phi r^{(s-3)/2} u\|_M),
\end{split}\end{equation*}
with $C_0$ independent of $c$ in the definition of $\phi$. We also have a
similar estimate for the time derivatives, using the PDE and the fact
that $t^{-1} \lesssim r^{-1}$ on $\supp \phi:$
\begin{equation}\begin{split}\label{eq:pa-t-est}
&\|t^{-(k+s)/2}  r^{(s-1)/2}\tilde\phi\phi Q_i u\|_M\\
&\leq C_1 \big(\|t^{-(k+s)/2} r^{(s-1)/2}\tilde\phi\phi\pa_t u\|_M+\|t^{-(k+s)/2} r^{(s-3)/2} \tilde\phi\phi u\|_M\big)\\
&\leq C_2 \Big(\|t^{-(k+s)/2} r^{(s-1)/2}\tilde\phi\phi\nabla_X u\|_M+\|t^{-(k+s)/2} r^{(s-3)/2} \tilde\phi\phi u\|_M\\
&\qquad\qquad+\|t^{-(k+s)/2} r^{(s+1)/2}\tilde\phi\phi(\Box+V) u\|_M\\
&\qquad\qquad+\|t^{-(k+s)/2} r^{(s-1)/2}\nabla_X u\|_{\supp d(\phi\tilde\phi)}+\|t^{-(k+s)/2} r^{(s-3)/2} \tilde\phi\phi u\|_{\supp d(\phi\tilde\phi)}\Big).
\end{split}\end{equation}
To show the validity of the last step, consider
the (indefinite!)\ Dirichlet form relative to $dt^2-g$, and use that
\begin{equation}\begin{split}\label{eq:D-form-1}
&\langle t^\alpha r^\beta \tilde\phi\phi\nabla u,t^\alpha r^\beta \tilde\phi\phi\nabla u\rangle_{M,dt^2-g}
-\langle t^\alpha r^{\beta+1} \tilde\phi\phi\Box u,t^\alpha r^{\beta-1} \tilde\phi\phi u\rangle_M\\
&=
\langle f_0 t^{\alpha-1}r^\beta \tilde\phi\phi u,t^\alpha r^\beta \tilde\phi\phi\pa_t u\rangle_M
+\langle f t^{\alpha}r^{\beta-1} \tilde\phi\phi u,t^\alpha r^\beta \tilde\phi\phi\nabla_X u\rangle_M\\
&\qquad\qquad+\sum_j\langle t^{\alpha-1}r^{\beta}Q'_ju,t^{\alpha}r^{\beta}R'_ju\rangle_M,
\end{split}\end{equation}
where $f_0$, $f$ and $Q'_j$ are zero'th order symbols on $\tilde M$,
$R'_j\in t^{-1}\Vb(\tilde M)$, and $Q'_j$ is
supported on $\supp d(\phi\tilde\phi)$ (so $R'_j$ can be taken supported nearby).
Indeed,
the terms on the right hand side can be controlled: the second term directly
by the positive spatial term via Cauchy-Schwarz,
\begin{equation}\label{eq:D-form-comm1}
|\langle f t^{\alpha}r^{\beta-1} \tilde\phi\phi u,t^\alpha r^\beta \tilde\phi\phi\nabla_X u\rangle_M|\lesssim\| t^{\alpha}r^{\beta-1} \tilde\phi\phi u\|_M^2+\|t^\alpha r^\beta \tilde\phi\phi\nabla_X u\|_M^2,
\end{equation}
while the first can be
Cauchy-Schwarzed with a small constant in front of the $\pa_t$ factor in the
pairing, which then
can be reabsorbed in the Dirichlet form, while the other factor can be controlled
directly from the spatial positive term using that $r/t$ is bounded, and indeed
small:
\begin{equation}\label{eq:D-form-comm2}
|\langle f_0 t^{\alpha-1}r^\beta \tilde\phi\phi u,t^\alpha r^\beta \tilde\phi\phi\pa_t u\rangle_M|
\lesssim \ep\|t^\alpha r^\beta \tilde\phi\phi\pa_t u\|_M^2+\ep^{-1}\|t^{\alpha-1}r^\beta \tilde\phi\phi u\|_M^2.
\end{equation}
In addition, the $\Box u$ term can be controlled using Cauchy-Schwarz:
\begin{equation}\label{eq:D-form-comm3}
\langle t^\alpha r^{\beta+1} \tilde\phi\phi\Box u,t^\alpha r^{\beta-1} \tilde\phi\phi u\rangle_M
\leq \|t^\alpha r^{\beta-1} \tilde\phi\phi u\|_M^2
+\|t^\alpha r^{\beta+1} \tilde\phi\phi\Box u\|_M^2.
\end{equation}
Finally, we change from $\Box$ to $\Box+V$ using that $V\in S^{-2}(X)$, so
\begin{equation}\label{eq:readd-V}
\| t^\alpha r^{\beta+1} \tilde\phi\phi Vu\|_M\lesssim
\|t^\alpha r^{\beta-1} \tilde\phi\phi u\|_M.
\end{equation}
Combining \eqref{eq:D-form-1}--\eqref{eq:readd-V},
\begin{equation}\begin{split}\label{eq:D-form-main}
&(1-\tilde C\ep)\| t^\alpha r^\beta \tilde\phi\phi\pa_t u\|_M^2\\
&\lesssim \|t^\alpha r^\beta \tilde\phi\phi\nabla_X u\|_M^2
+\ep^{-1}\|t^{\alpha-1}r^\beta \tilde\phi\phi u\|_M^2
+\big(\| t^{\alpha}r^{\beta-1} \tilde\phi\phi u\|_M^2+\|t^\alpha r^\beta \tilde\phi\phi\nabla_X u\|_M^2\big)\\
&\qquad\qquad+\big(\| t^{\alpha-1}r^{\beta}  u\|^2_{\supp d(\phi\tilde\phi)}
+\|t^\alpha r^\beta \tilde\phi\phi\nabla_X u\|^2_{\supp d(\phi\tilde\phi)}\big)\\
&\qquad\qquad+\|t^\alpha r^{\beta+1} \tilde\phi\phi(\Box+V) u\|_M^2,
\end{split}\end{equation}
where $\ep>0$ can be arbitrarily chosen without affecting $\tilde C$, so in particular
we may assume $\tilde C\ep<1/2$. This proves
the second inequality in \eqref{eq:pa-t-est}.

Since
\begin{equation}\begin{split}\label{eq:small-c-est}
&|\langle Q_i^*\tilde\phi \phi r^s t^{-1} t^{-(k+s)} R_iu,u\rangle_M|\\
&\qquad\leq 2c\|t^{-(k+s)/2} r^{(s-1)/2}\tilde\phi\phi R_i u\|_M\|t^{-(k+s)/2} r^{(s-1)/2}\tilde\phi \phi Q_i u\|_M\\
&\qquad\leq c\|t^{-(k+s)/2} r^{(s-1)/2}\tilde\phi\phi R_i u\|_M^2+
c\|t^{-(k+s)/2} r^{(s-1)/2}\tilde\phi \phi Q_i u\|_M^2,
\end{split}\end{equation}
with $c$ arising from the extra factor of $r/t,$ which is $\leq 2c$ on $\supp\phi$,
in the pairing
over what is included in the
two factors on the right,
for sufficiently small $c>0$ these can be absorbed in the positive
definite spatial term in
\eqref{comm3}, modulo terms supported on $\supp d(\phi\tilde\phi)$ and
modulo terms involving $(\Box+V) u$.

It remains to check that the ``quantum'' term (in that we need global estimates
on $\tf$ to control it), $E_s$, in \eqref{comm3} can be controlled.
With $\rho_0\in\CI_c(\RR)$ identically $1$ on a neighborhood of $\supp\phi_0$,
supported in $(-3,3)$, and
let $\rho(\dummyvar)=\rho_0(\dummyvar/c)$, $c$ as above, and consider $\rho=\rho(r/t)$.
We write
\begin{equation*}
|\langle r^{s-4}\phi(r/t)E_s\phi(r/t)u,u\rangle_X|
\lesssim
\|r^{(s-4)/2}\phi(r/t)\nabla_X u\|^2_X+\|r^{(s-4)/2}\phi(r/t)u\|^2_X.
\end{equation*}
We remark that
the left hand side rearranges the factors from \eqref{comm3} by commuting
a factor of $\phi$ through $E_s$; the difference is supported in $\supp d\phi$,
and when multiplied by $t^{-k-s}$, it is in $t^{-k-5}\Diffb^2(\tilde M)$, i.e.\ can
be controlled the same way $F_{k,s}$ was (and is indeed better behaved).
Now write
\begin{equation*}
u=\tilde\psi(H^2(\Delta+V))\rho(r/t) u+\tilde\psi(H^2(\Delta+V)) (1-\rho(r/t)) u,
\end{equation*}
and estimate the summands (and their $X$-gradients) individually.

First, to deal with the $\tilde\psi(H^2(\Delta+V)) (1-\rho(r/t))u$ term,
note that by Proposition~\ref{prop:weighted-localized-spectral-cutoff}, for any $N$,
\begin{equation*}\begin{split}
&\|\phi(r/t)\tilde\psi(H^2(\Delta+V))(1-\rho(r/t))r^N\|_{\cL(L^2(X),L^2(X))}\\
&\qquad+\|\nabla_X\phi(r/t)\tilde\psi(H^2(\Delta+V))(1-\rho(r/t))r^N\|_{\cL(L^2(X),L^2(X))}\leq C_{H,N} t^{-N},
\end{split}\end{equation*}
where the constant $C_{H,N}$ depends on $H$ (and $N$).
Since $u$ is tempered, for sufficiently large $N$, the space-time norm of
$r^{-N}t^{-N}u$ is bounded (with the bound depending on $H$), controlling this term.

Next, to deal with the
$$
\tilde u=\tilde\psi(H^2(\Delta+V)) \rho(r/t)u
$$
term, we use the low
energy localization, i.e.\ we take $H$ large.
First, for $\ep'\in(0,1/2)$,
\begin{equation*}\begin{split}
&\| r^{(s-4)/2}\phi(r/t)\tilde u\|^2_X
\leq \|r^{(s-4)/2}\tilde u\|^2_X\leq \|r^{(s-3)/2-\ep'}\tilde u\|^2_X,\\
&\|r^{(s-4)/2}\phi(r/t)\nabla_X\tilde u\|^2_X\leq \|r^{(s-3)/2-\ep'}\nabla_X\tilde u\|^2_X.
\end{split}\end{equation*}
By the Hardy/Poincar\'e lemma,
Lemma~\ref{lemma1}, using $0<s<1$, and choosing $\ep'\in(0,1/2)$ such that
$-(s-3)/2+\ep'<3/2$,
\begin{equation}\begin{split}\label{eq:E-s-est}
&\| r^{(s-3)/2-\ep'}\tilde u\|^2_X+\| r^{(s-3)/2-\ep'}\nabla_X\tilde u\|^2_X\\
&\lesssim H^{-2\ep'}\|r^{(s-3)/2} \rho(r/t) u\|_X^2\\
&=
H^{-2\ep'}\big(\|\phi(r/t)r^{(s-3)/2}  u\|_X^2+\|(\rho(r/t)-\phi(r/t))r^{(s-3)/2}  u\|_X^2\big).
\end{split}\end{equation}
For $H$ sufficiently large, the first term on the right hand side of
\eqref{eq:E-s-est}
can be absorbed into the zero'th order
positive definite spatial term,
\begin{equation}\label{eq:pos-def-spatial-1}
\Big\langle \phi(r/t)^2
\big(2s\pa_r^*r^{s-1}\chi(r)\pa_r+2r^{s-3}\chi(r)\Lap_Y+\frac{n-1}{2}\,(1-s)(n+s-3)r^{s-3}\big)u,u\Big\rangle_X,
\end{equation}
modulo errors supported on $\supp d(\tilde\phi\phi)$, which are of the same
kind as those in $F_{k,s}$ in \eqref{comm3}. On the other hand,
the second term on the right hand side of \eqref{eq:E-s-est}
is supported on
$\supp(\rho-\phi)$, and is thus similar to the $F_{k,s}$ terms, with slightly
larger support (but still within the region where this proposition gives
no improvement upon the a priori assumption). Explicitly, when also integrated
in $t$,
on $\supp (\rho-\phi)$ we have
$\nabla u \in t^\kappa L^2(M)$ (with $(k+1)/2=\kappa$), and
since on $\supp(\rho-\phi)$, $r\sim t$,
\begin{equation}\begin{split}\label{eq:chi-phi-est}
&\int t^{-k-s}\tilde\phi(t)^2\|(\rho(r/t)-\phi(r/t))r^{(s-3)/2} u\|_X^2\,dt\\
&\qquad\lesssim \int t^{-k-s}\tilde\phi(t)^2\|t^{(s-3)/2}(\rho(r/t)-\phi(r/t)) u\|^2_X\,dt\\
&\qquad\lesssim \| t^{-(k+1)/2}  \tilde\phi(t)^2 (\rho(r/t)-\phi(r/t)) u\|^2_M.
\end{split}\end{equation}

Now, for $\gamma>0$,
\begin{equation*}\begin{split}
&|\langle B_{k,s}u,(\Box+V) u\rangle_M|=|\langle B_{k,s}u,\rho(r/t)(\Box+V) u\rangle_M|\\
&\leq \|r^{(s+1)/2}t^{-(k+s)/2}\rho(r/t)(\Box+V) u\|_M \|r^{-(s+1)/2}t^{(k+s)/2}B_{k,s}u\|_M\\
&\leq \gamma^{-1}\|r^{(s+1)/2}t^{-(k+s)/2}\rho(r/t)(\Box+V) u\|_M^2
+\gamma \|r^{-(s+1)/2}t^{(k+s)/2}B_{k,s}u\|_M^2\\
&\leq \gamma^{-1}\|r^{(s+1)/2}t^{-(k+s)/2}\rho(r/t)(\Box+V) u\|_M^2
+\gamma \Big(\|r^{(s-1)/2}t^{-(k+s)/2}\rho(r/t)\chi(r)\pa_r u\|_M^2\\
&\qquad\qquad+\|r^{(s-3)/2}t^{-(k+s)/2}\rho(r/t) u\|_M^2
+\|t^{-(k+s)/2}(1-\chi(r))\rho(r/t)\nabla_X u\|_M^2\Big).
\end{split}\end{equation*}
Taking $\gamma>0$ sufficiently small, the second term on the right hand
side can be absorbed into the positive spatial terms in \eqref{comm3}, modulo terms
that are supported outside $\supp (1-\phi\tilde\phi)$, but in that region
we have a priori control of these terms.
We thus deduce from \eqref{comm3}
that
\begin{equation}\begin{split}\label{eq:commutator-result}
&\|t^{-(k+s)/2} r^{(s-1)/2} \tilde\phi\phi\nabla_X u\|_M+\|t^{-(k+s)/2} r^{(s-3)/2} \tilde\phi\phi u\|_M\\
&\lesssim \|r^{(s+1)/2}t^{-(k+s)/2}(\Box+V) u\|_{L^2(\tOmega)}+\|u\|_{t^{(k+3)/2}\Hb^1(\Upsilon)},
\end{split}\end{equation}
where $\Upsilon=\supp d\phi \cup \supp (\rho-\phi)$ and with the last
space being $\Hb^1(\tilde M)$ on $\Upsilon$, provided $0<s<1$ and
provided the right hand side is finite. Recall that
Lemma~\ref{lemma:dphihardy} means that the second term on
the right hand side may be
replaced by
$$
\norm{\nabla u}_{t^{(k+1)/2}L^2(\Upsilon)}.
$$
In view
of the Dirichlet form estimate, \eqref{eq:D-form-main},
we also get the bound for the
time derivative which is analogous to \eqref{eq:commutator-result}. Noting that
\begin{equation*}\begin{split}
&\|t^{-(k+s)/2} r^{(s-1)/2} \tilde\phi\phi\nabla u\|_M+\|t^{-(k+s)/2} r^{(s-3)/2} \tilde\phi\phi u\|_M\\
&\qquad=\|t^{-(k+1)/2} (t/r)^{(1-s)/2}\tilde\phi\phi\nabla u\|_M
+\|t^{-(k+1)/2} (t/r)^{(1-s)/2} r^{-1}\tilde\phi\phi u\|_M,
\end{split}\end{equation*}
and $0<s<1$, so $0<(1-s)/2<1/2$,
while (using $r/t<1$ on $\tOmega$ and $s>0$)
\begin{equation*}
\|r^{(s+1)/2}t^{-(k+s)/2}(\Box+V) u\|_{L^2(\tOmega)}
\lesssim \|r^{1/2}t^{-k/2}(\Box+V) u\|_{L^2(\tOmega)},
\end{equation*}
this proves the estimate of the
theorem, provided the formal pairings employed above are finite
and hence the computation justified.

In order to employ the above argument
without a priori assumptions on $u$, we need to introduce
a weight, $(1+\alpha t)^{-1}$, $\alpha>0$, in $B_{k,s}$
to justify the arguments, and let
$\alpha\to 0$. Roughly speaking, such a weight does not cause problems
since we could deal with {\em arbitrary} weights $t^{-k}$ already (i.e.\ $k$ did
not need to have a sign). More precisely,
the contribution of this weight to the commutator of $(1+\alpha t)^{-1} B_{k,s}$
and $\Box$
is similar to
\eqref{eq:t-weight-comm}, namely
$[\pa_t^2,t^{-k-s}(1+\alpha t)^{-1}]\in t^{-k-s-3}\Diffb^1(\overline{\RR}_t)$ for
$\alpha>0$, and it is uniformly bounded in the larger space
$t^{-k-s-2}\Diffb^1(\overline{\RR}_t)$ for $\alpha\in [0,1)$.
Thus, its role is analogous to that of the $Q_i^*R_i$ terms
of \eqref{comm3}, which in turn
arose from \eqref{eq:t-weight-comm}; these are
estimated in \eqref{eq:small-c-est}, and can be controlled by
making $c$ small. Letting $\alpha\to 0$,
and using standard functional analytic arguments, completes the proof of the
theorem.\qed

We finally remark that a version of this argument also works for
solutions of the wave equation at intermediate frequencies, via Mourre
estimate techniques (see especially \cite{FroMourre}).  While this is
not interesting for the homogeneous wave equation, whose middle-frequency
solutions will have rapid decay inside the light cone, for the
inhomogeneous equation it may be of some interest.  The key point then
is that one works with $\psi(\Delta+V)$, and $\tilde\psi(\Delta+V)$,
supported near a fixed energy $\lambda_0>0$, and one controls the
analogue of the $E_s$ term (that arose above) by shrinking the support
of $\tilde\psi$, using that $\tilde\psi(\Delta+V)E_s$ converges to $0$
in norm as the support is shrunk to $\{\lambda_0\}$ by the relative
compactness of $E_s$ (which follows from elliptic estimates and its
decay at spatial infinity). This holds since $\tilde\psi(\Delta+V)$
converges to $0$ strongly as the support is shrunk to $\{\lambda_0\}$,
since $\lambda_0$ is not an $L^2$ eigenvalue of $\Delta+V$.

\end{document}